\documentclass[12pt]{article}
\usepackage{a4wide}

\usepackage{amsmath,amsthm}
\usepackage{amsfonts}
\usepackage{amscd}
\usepackage{mathrsfs}
\usepackage{amssymb}
\let\mathscr\undefined

\renewcommand{\le}{\leqslant}
\renewcommand{\leq}{\leqslant}

\renewcommand{\ge}{\geqslant}
\renewcommand{\geq}{\geqslant}

\usepackage{graphicx}
\usepackage{dsfont}
\usepackage{color}
\usepackage{bbm}
\usepackage[colorlinks]{hyperref}
\hypersetup{citecolor=blue}

\usepackage{sseq}
\usepackage{tikz-cd}
\usepackage[abs]{overpic} 
\usepackage{bm}

\newcommand*{\longhookrightarrow}{\ensuremath{\lhook\joinrel\relbar\joinrel\rightarrow}}

\theoremstyle{theorem}
\newtheorem{TheoremA}{Theorem}

\newtheorem{CorollaryA}{Corollary}

\newtheorem{theorem}{Theorem}[section]
\newtheorem{lemma}[theorem]{Lemma}

\newtheorem*{question*}{Question}

\theoremstyle{definition}
\newtheorem*{definition*}{Definition}
\newtheorem{definition}[theorem]{Definition}

\newtheorem*{example*}{Example}
\newtheorem{example}[theorem]{Example}
\newtheorem*{observation*}{Observation}

\newtheorem*{Goal*}{Goal}

\newtheorem*{Assumption*}{Assumption}

\theoremstyle{remark}
\newtheorem*{remark*}{Remark}
\newtheorem{remark}[theorem]{Remark}

\numberwithin{equation}{section}

\newcommand{\ow}{\omega}
\newcommand{\p}{\partial}

\newcommand{\C}{{\mathbb{C}}}
\newcommand{\R}{{\mathbb{R}}}
\newcommand{\Q}{{\mathbb{Q}}}
\newcommand{\Z}{{\mathbb{Z}}}
\newcommand{\N}{{\mathbb{N}}}

\DeclareMathOperator{\morse}{Morse}
\DeclareMathOperator{\floer}{Floer}

\DeclareMathOperator{\topo}{top}
\DeclareMathOperator{\Vol}{Vol}

\DeclareMathOperator{\anti}{anti}
\DeclareMathOperator{\inv}{inv}
\DeclareMathOperator{\rel}{rel}
\DeclareMathOperator{\ch}{ch}

\DeclareMathOperator{\bottom}{bottom}
\DeclareMathOperator{\topp}{top}

\DeclareMathOperator{\supp}{supp}
\DeclareMathOperator{\Gr}{Gr}

\DeclareMathOperator{\Diff}{Diff}

\DeclareMathOperator{\im}{Im}

\DeclareMathOperator{\Symp}{Symp}

\DeclareMathOperator{\Spec}{Spec}

\DeclareMathOperator{\Fix}{Fix}

\DeclareMathOperator{\HF}{HF}
\DeclareMathOperator{\CF}{CF}
\DeclareMathOperator{\CM}{CM}
\DeclareMathOperator{\HM}{HM}
\DeclareMathOperator{\Ho}{H}
\DeclareMathOperator{\st}{st}
\DeclareMathOperator{\Int}{int}

\begin{document}

\title{
Floer-theoretic entropy of exact symplectomorphisms
}

\author{Joontae Kim and Myeonggi Kwon}

\maketitle

\begin{abstract}
We introduce the notion of a Penner-type class of an  exact symplectomorphism on a Liouville domain with an $A_k$-configuration of Lagrangian spheres for $k\ge 2$, and prove that such a class has positive Floer-theoretic entropy.
As a corollary, we construct infinitely many smoothly trivial symplectic isotopy classes of exact symplectomorphisms with positive Floer-theoretic entropy on any $4n$-dimensional Liouville domain that admits an $A_2$-configuration of Lagrangian spheres.
We also prove that the Floer-theoretic entropy of an exact symplectomorphism on a Liouville domain provides a lower bound for its topological entropy.
\end{abstract}


\section{Introduction}\label{sec: intro}
A symplectomorphism $\phi\colon W\to W$ of a Liouville domain $(W,\lambda)$ is called \emph{exact} if the 1-form $\phi^*\lambda-\lambda$ is exact.
We denote by $\Symp(W,\lambda)$ the \emph{group of exact symplectomorphisms} of $W$ supported in the interior $\Int W=W\setminus \p W$, that is,
$
\supp(\phi)=\overline{\{x\in W \mid \phi(x)\ne x\}} \subset \Int W,
$
and we endow it with the $C^\infty$-topology.
The \emph{symplectic mapping class group} of a Liouville domain $(W,\lambda)$ is defined as $\pi_0 \Symp(W,\lambda)$.
This group has a distinguished subgroup called the \emph{smoothly trivial symplectic mapping class group} defined by
$$
K(W,\lambda):=\ker\big[\pi_0 \Symp(W,\lambda) \to \pi_0 \Diff(W)\big],
$$
where $\Diff(W)$ denotes the group of diffeomorphisms of $W$.
It has been a central topic in symplectic topology to understand how non-trivial or large this group can be.
For results on \emph{closed} symplectic four manifolds $(M,\ow)$ (with the  analogous definition of $K(M,\ow)$), we refer the reader to \cite{Se00, Se08a, To15}.
Remarkably, it was shown in \cite{ShSm20, Sm22} that $K(X,\ow)$ is infinitely generated for certain K3 surfaces $(X,\ow)$.
In the context of Liouville domains $(W,\lambda)$, we refer to \cite{Se99, Ke14, BarGeiZeh19}.
In particular, \cite[Theorem~1.4]{BarGeiZeh19}, combined with the arguments in the first part of Section~\ref{sec: corollaryA}, implies that $K(W,\lambda)$ is infinite whenever $(W,\lambda)$ admits a Lagrangian sphere and $\dim W=4n$.

In this paper we study the smoothly trivial symplectic mapping class group $K(W,\lambda)$ of a Liouville domain from the perspective of the Floer-theoretic entropy.
While the algebraic ``size" of $K(W,\lambda)$ has been extensively studied, much less is known about the dynamical aspects of its classes. 
To formulate our main results, we begin by introducing the relevant notions.

The \emph{Floer-theoretic entropy} $h_{\floer}(\phi)$ of $[\phi]\in\pi_0\Symp(W,\lambda)$ (see Definition~\ref{def: floer_entropy}) measures the exponential growth rate of the dimension of the fixed point Floer homology $\HF(\phi)$ under iteration of $\phi$.
The \emph{Dehn twist} $\tau_L$ along a Lagrangian sphere $L$ in a symplectic manifold is a certain symplectomorphism  supported in a neighborhood of $L$, generalizing the classical Dehn twist on a surface.
It is uniquely defined up to symplectic isotopy and the choice of parametrisation $S^n\hookrightarrow L$.
If the ambient symplectic manifold is exact, then so is $\tau_L$.
We refer to \cite[16a, 16c]{Se08b}, \cite[Section~6]{Se99}, and \cite[Section~3.4]{To15} for details.
Unless we require a specific parametrisation, we assume that the parametrisation of any Lagrangian sphere is implicitly fixed whenever a Dehn twist along it is considered.
A collection $(L_1,\dots,L_k)$ of pairwise transverse Lagrangian spheres in a symplectic manifold satisfying
$$
|L_i\cap L_j|=\begin{cases}
	1 & \text{for $|i-j|=1$}, \\
	0 & \text{for $|i-j|\ge 2$}
\end{cases} 
$$
is called an \emph{$A_k$-configuration}.
In particular, two Lagrangian spheres $(L_1,L_2)$ transversely intersecting at a single point form an $A_2$-configuration.
We say that $[\phi]\in \pi_0\Symp(W,\lambda)$ is of \emph{$A_k$-Penner-type} with respect to an $A_k$-configuration $(L_1,\dots,L_k)$ if it has a representative $\phi$ that is a product of positive/negative Dehn twists $\tau_{L_i}^{\pm1}$ along $L_i$ for $i$ odd and negative/positive Dehn twists $\tau_{L_j}^{\mp1}$ along $L_j$ for $j$ even, such that all Lagrangian spheres $L_1,\dots,L_k$ appear at least once.
The motivation for this definition is discussed in Remark~\ref{rem: akpenner}.
For instance, if $k=3$, then the class of $\phi=\tau_{L_1}^3\tau_{L_2}^{-2}\tau_{L_1}\tau_{L_3}^5$ is of $A_3$-Penner-type.
If $[\phi]\in\pi_0\Symp(W,\lambda)$ is of $A_2$-Penner-type, then $[\phi]$ is represented by the map $\phi_V$ in \eqref{eq: A2_pennerLag}, viewed as a map in $\Symp(W,\lambda)$.

Our main result provides a sufficient condition for the existence of symplectic isotopy classes of positive Floer-theoretic entropy:

\begin{TheoremA}\label{thm: mainthmA}
Let $(W,\lambda)$ be a Liouville domain with an $A_k$-configuration of Lagrangian spheres for $k\ge 2$.
Then, every $A_k$-Penner-type class $[\phi]\in \pi_0\Symp(W,\lambda)$ satisfies $h_{\floer}(\phi)>0$.
\end{TheoremA}
It is known that every Milnor fibre of an isolated hypersurface singularity with Milnor number at least $2$ admits an $A_2$-configuration of Lagrangian spheres, see \cite[Lemma~9.9 and Lemma~9.11]{Ke14}.
Hence, there is an abundance of examples to which our main result applies.
Theorem~\ref{thm: mainthmA} may fail if the Liouville domain is replaced by a closed symplectic manifold, as illustrated in Example~\ref{ex: failure}.
We note that a product of Dehn twists along a \emph{single} Lagrangian sphere, that is, $A_1$-configuration, does not produce positive Floer-theoretic entropy, see Example~\ref{ex: singledehn}.

We outline the proof of Theorem~\ref{thm: mainthmA} in the four dimensional case, $\dim W=4$.
If $[\phi]\in\pi_0\Symp(W,\lambda)$ is of $A_k$-Penner-type, then we select a representative $\phi$ supported in a neighborhood of the $A_k$-configuration, identified with a Milnor fibre $V$ of the $A_k$-singularity.
This Milnor fibre admits a symplectic involution $\sigma$ whose fixed locus $S$ is a two dimensional Milnor fibre of the $A_k$-singularity.
We may assume that $\phi|_V$ commutes with $\sigma$.
The main ingredients of the proof of Theorem~\ref{thm: mainthmA} are now the local-to-global identity (Theorem~\ref{thm: localtoglobal}) and a Smith-type inequality for Floer-theoretic entropy (Theorem~\ref{thm: milnor_smith}):
the former implies $h_{\floer}(\phi)=h_{\floer}(\phi|_V)$ and the latter yields $h_{\floer}(\phi|_V)\ge h_{\floer}(\phi|_S)$.
Since $[\phi]$ is of $A_k$-Penner-type, $\phi|_S$ corresponds to Penner's construction (Theorem~\ref{thm: penner_construction}), which is pseudo-Anosov, and hence $h_{\floer}(\phi|_S)>0$ by Cotton-Clay \cite{Co09}, see Theorem~\ref{thm: cotton_clay}.

Theorem~\ref{thm: mainthmA} yields the following application:


\begin{CorollaryA}\label{cor: corollaryA}
In the situation of Theorem~\ref{thm: mainthmA} with $\dim W=4n$ and $k=2$, there exist infinitely many symplectic isotopy classes in $K(W,\lambda)$ with positive Floer-theoretic entropy.
\end{CorollaryA}

In fact, these infinitely many symplectic isotopy classes are generated by taking powers of any \emph{single} $A_2$-Penner-type class.
Since the positivity of Floer-theoretic entropy is not necessarily preserved under taking powers of a class, Corollary~\ref{cor: corollaryA} is not an immediate consequence.
Instead, we observe that all powers of an $A_2$-Penner-type class remain of $A_2$-Penner-type.
\begin{remark}
One can construct even more classes in $K(W,\lambda)$ with positive Floer-theoretic entropy.
For example, consider $\phi=\tau_{L_1}^8\tau_{L_2}^{-8}$ and $\psi=\tau_{L_1}^{16}\tau_{L_2}^{-16}$, whose classes in $K(W,\lambda)$ have infinite order.
As shown in Section~\ref{sec: corollaryA}, the condition that the powers of the Dehn twists are multiples of 8 ensures that $[\phi]$ and $[\psi]$ are smoothly trivial.
Through a direct computation combined with Theorem~\ref{thm: mainthmA}, one can show that $[\phi]^\ell\ne [\psi]^k$ for any $\ell,k\ge 1$.
\end{remark}

The second goal of this paper is to prove the following theorem.
Let $h_{\topo}(\phi)$ denote the \emph{topological entropy} of $\phi\in\Diff(W)$, as defined in \cite[Definition~3.1.3]{KaHa95} and \cite[Section A.2]{AbAlSaSc23}.
\begin{TheoremA}\label{thm: mainthmB}
Let $(W,\lambda)$ be a Liouville domain.
Then, every $\phi\in \Symp(W,\lambda)$ satisfies
\begin{equation}\label{eq: entropy_ineq}
h_{\topo}(\phi)\ge h_{\floer}(\phi).	
\end{equation}
\end{TheoremA}
The proof relies on arguments using a Lagrangian tomograph and Crofton's inequality, see \cite[Section~5.2]{CiGiGu21}, \cite[Section~3]{BaLe25}, and  \cite[Section~4.3]{KiKw26}.
Theorem~\ref{thm: mainthmB} implies that classes of Penner-type necessarily have positive topological entropy, which is not obvious for topological reasons.
The invariance property of Floer-theoretic entropy ensures that the right-hand side of Equation~\ref{eq: entropy_ineq} is independent of the choice of a representative $\phi$ of $[\phi]\in\pi_0\Symp(W,\lambda)$.
Another intriguing feature is that the uniform positivity of topological entropy holds for \emph{symplectic mapping classes} having positive Floer-theoretic entropy, rather than depending on a choice of a representative. 
Consequently, this exhibits a genuine symplectic phenomenon.
Regarding analogous results, we refer to \cite{Po02,FrSc05,KKL18,AlMe19,Da20} and references therein.
\begin{remark}\label{rem: intro}
Since $h_{\topo}(\phi)<\infty$ for every diffeomorphism $\phi$ on a compact smooth manifold \cite[Corollary~3.2.10]{KaHa95}, we deduce that $h_{\floer}(\phi)<\infty$ for any $\phi\in\Symp(W,\lambda)$.
Note that this finiteness is not a consequence of the definition of Floer-theoretic entropy.
\end{remark}
\begin{remark}
We note that \cite{GWX26} has previously studied the positive topological entropy of certain Penner-type maps from $A_k$-configurations, as well as the exponential relative symplectic growth rate in terms of the geometric setup described in Section~\ref{sec: milnorfibres}.
\end{remark}

\subsection*{Organization of the paper}
Section~\ref{sec: floertheory} recalls the relevant Floer theory, specifically, fixed point and Lagrangian Floer homology, and introduces Floer-theoretic entropy with its basic properties.
The local-to-global identity for this entropy, which is one of the core ingredients for the proof of Theorem~\ref{thm: mainthmA}, is derived in Section~\ref{sec: LtoG} through the construction of an exact triangle involving restriction of certain twisted Floer trajectories. 
Section~\ref{sec: smith_ineq} proves a Smith-type inequality for fixed point Floer homology, based on the work of Seidel--Smith \cite{SS}.
Subsequently, Section~\ref{sec: milnorfibre} reviews the basic geometry of a Milnor fibre of the $A_k$-singularity, following \cite[Section~6c]{KhSe02}, and introduce the notion of Penner-type symplectomorphisms, motivated by Penner's result on pseudo-Anosov surface diffeomorphisms \cite{Pe88}.
Finally, the proofs of Theorems~\ref{thm: mainthmA}, \ref{thm: mainthmB}, and Corollary~\ref{cor: corollaryA} are provided in Section~\ref{sec: proofs}.
\section{Floer theory}\label{sec: floertheory}
Throughout this section, $(W,\lambda)$ is a Liouville domain, meaning that $W$ is a compact smooth manifold with boundary, and $\lambda$ is a 1-form on $W$ such that $d\lambda$ is symplectic, and the Liouville vector field $Z$ of $\lambda$ defined by $\iota(Z)d\lambda=\lambda$ points outwards along $\p W$.
The Liouville flow of $Z$ provides a collar embedding $((1-\epsilon,1]\times \p W)\subset W$, yielding there cylindrical coordinates $(r,y)$.

\subsection{Fixed point Floer homology and Floer-theoretic entropy}\label{sec: fixedfloer}
We begin by recalling the definition of the (fixed point) Floer homology $\HF(\phi)$ of $\phi\in \Symp(W,\lambda)$, which is an invariant of $\phi$ up to isotopy.
For a detailed construction, we refer to \cite{Ul17}, \cite[Section~4]{Se01}, \cite[Section~2.2]{Mc12}, and \cite[Section~4]{Mc19}.

Let ${\bm J}=\{J_t\}_{t\in \R}$ be a family of compatible almost complex structures on $W$ that are cylindrical near the boundary, meaning that in cylindrical coordinates $(r,y)\in (1-\epsilon,1]\times \p W$, the condition $J_t^*\lambda=dr$ holds. 
Additionally, suppose that $\bm{J}$ satisfies the periodicity condition
$$
\phi^*J_{t+1}=J_t.
$$
We denote by $\mathcal{J}_\phi$ the space of such families of almost complex structures, and we call ${\bm J}\in \mathcal{J}_\phi$ \emph{admissible}.
Next, consider a non-degenerate Hamiltonian $H\colon \R \times W \to \R$ that is linear near the boundary, that is, $H=ar+b$ in cylindrical coordinates on $(1-\epsilon,1]\times \p W$, with positive slope strictly less than $\min \Spec(\p W, \lambda|_{\p W})$, and $H_t=H(t,\cdot)$ satisfies the periodicity condition
$$
\phi^*H_{t+1}=H_t.
$$
Here, $\Spec(\p W,\lambda|_{\p W})$ denotes the set of the periods of closed Reeb orbits, and by a \emph{non-degenerate} Hamiltonian $H$ (where $\phi_H^t$ denotes its Hamiltonian flow) we mean $\psi_H:=\phi_H^{-1}\circ \phi$ is non-degenerate, that is, for every fixed point $x$ of $\psi_H$ the linearized return map $D_x\psi_H$ does not have 1 as an eigenvalue.
The space of such Hamiltonians is denoted by $\mathcal{H}_{\phi}$, and we refer to any $H\in \mathcal{H}_\phi$ as \emph{admissible}.
We abbreviate by
$$
\mathcal{P}(H)=\{x\colon \R\to W \mid \dot{x}(t)=X_{H_t}(x(t)),\ \phi(x(t))=x(t+1)\}
$$
the set of \emph{Hamiltonian twisted orbits} of $H$, which corresponds bijectively to $\Fix(\psi_H)$, the set  of fixed points of $\psi_H$.

Given $x_\pm \in \mathcal{P}(H)$ with $x_-\ne x_+$,
a \emph{twisted Floer trajectory} for $(H,{\bm J})$ from $x_-$ to $x_+$ is a smooth map $u\colon \R\times \R \to W$ satisfying
$$
\begin{cases}
\p_s u + J_t(u)(\p_t u - X_{H_t}(u))=0, \\
\phi(u(s,t))=u(s,t+1), \\
\displaystyle \lim_{s\to \pm \infty}u(s,t)=x_\pm(t).
\end{cases}
$$
If $\phi=\mathds{1}$ is the identity, a twisted Floer trajectory $u$ is called \emph{untwisted}.
For generic ${\bm J}\in \mathcal{J}_\phi$, the associated moduli space defined by
$$
\mathcal{M}(x_-,x_+;H,{\bm J})=\big\{ \text{all twisted Floer trajectories for $(H,{\bm J})$ from $x_-$ to $x_+$}\big\}\, \big/\, \R,
$$
where the quotient is taken with respect to translation in the $s$-variable, is a smooth manifold, and its 0-dimensional component $\mathcal{M}^0(x_-,x_+;H,{\bm J})$ is compact.
The \emph{Floer chain complex} of $\phi$ with respect to $(H,{\bm J})$ is the $\Z_2$-vector space generated by the orbits of $\mathcal{P}(H)$,
$$
\CF(\phi;H)=\bigoplus_{x\in \mathcal{P}(H)}\Z_2\langle x \rangle
$$
whose differential $\p^{\bm J}\colon \CF(\phi;H)\to \CF(\phi;H)$ is defined by counting rigid twisted Floer trajectories
$$
\p^{\bm J}(x_-)=\sum_{x_+\in \mathcal{P}(H)} \#_2\mathcal{M}^0(x_-,x_+;H,{\bm J})\cdot x_+.
$$
The homology $\HF(\phi;H,{\bm J})$ of the Floer chain complex is called the \emph{(fixed point) Floer homology} of $\phi$.
A standard continuation argument shows that this homology does not depend on the choice of a  generic pair $(H,{\bm J})$, so we simply write $\HF(\phi)$.
Since every path in $\Symp(W,\lambda)$ can be realized by an isotopy generated by a Hamiltonian supported in the interior of $W$, the following invariance property holds: if $\phi,\psi$ are isotopic in $\Symp(W,\lambda)$, that is, $[\phi]=[\psi]\in\pi_0\Symp(W,\lambda)$, then $\HF(\phi)\cong \HF(\psi)$, see \cite[Theorem~2.34]{Ul17}.
\begin{definition}\label{def: floer_entropy}
	The \emph{Floer-theoretic entropy} of $[\phi]\in \pi_0\Symp(W,\lambda)$ is defined by
	$$
	h_{\floer}(\phi)=\limsup_{k\to \infty}\frac{\log \dim \HF(\phi^k)}{k}\in [0,\infty].
	$$
\end{definition}
We refer to \cite{Sm12, Fe12} for a study of this invariant in different contexts.
Note that $h_{\floer}(\phi)$ does not depend on the choice of representative in the isotopy class $[\phi]\in \pi_0\Symp(W,\lambda)$, and it turns out that $h_{\floer}(\phi)$ is always finite, see Remark~\ref{rem: intro}.
The significance of the positivity of $h_{\floer}(\phi)$ can be highlighted within the context of the mapping class group as follows:
\begin{lemma}\label{lem: entropy_infinite}
If $h_{\floer}(\phi)>0$ for $[\phi]\in \pi_0\Symp(W,\lambda)$, then $[\phi]$ has infinite order in $\pi_0\Symp(W,\lambda)$.
\end{lemma}
\begin{proof}
By the assumption, we can pick a sequence $i_1<i_2<\cdots$ of positive integers such that for any $k\in \N$ we have
$$
\dim \HF(\phi^{i_k})<\dim \HF(\phi^{i_{k+1}}).
$$
Arguing by contradiction, suppose that there exists $\ell>0$ such that $\phi^\ell$ is isotopic to the identity.
For each $k\in \N$, we can write $i_k=\ell p_k +q_k$ for some $p_k\ge 0$ and $0\le q_k<\ell$.
It follows that $\phi^{i_k}$ is isotopic to $\phi^{q_k}$, and hence $\dim \HF(\phi^{i_k})=\dim \HF(\phi^{q_k})$ by the invariance property.
However, this leads a contradiction: the sequence $\{q_k\}$ takes values in the finite set $\{0,1,\dots,\ell-1\}$, implying that $\dim \HF(\phi^{i_k})$ can only take finitely many distinct values, whereas $\dim \HF(\phi^{i_k})$ is strictly increasing.
\end{proof}
\begin{remark}\label{rem: linear_infinite}
The result in Lemma~\ref{lem: entropy_infinite} continues to hold when $[\phi]\in\pi_0\Symp(W,\lambda)$ has linear growth rate, namely $\displaystyle\kappa(\phi):=\limsup_{k\to \infty}\frac{1}{k}\dim \HF(\phi^k) \in (0,\infty)$.
\end{remark}

\subsection{Lagrangian Floer homology and canonical isomorphism}\label{sec: lagrfloer}
We will make use of a canonical isomorphism between the Floer homology of $\phi\in\Symp(W,\lambda)$ and Lagrangian Floer homology to prove the Smith-type inequality for Floer homology, see Theorems~\ref{thm: can_isom} and \ref{thm: smith_floer}.
Since Lagrangian Floer theory for admissible Lagrangians is standard, we recall it only briefly.
For details of the construction and basic properties, we refer to the accounts \cite[Section~5]{KhSe02}, \cite[Section~3.1]{SS}, and \cite[Section~2.1]{KKL18}.

Recall that a Lagrangian $L\subset W$ is called \emph{admissible} if $L$ intersects $\p W$ transversely and the 1-form $\lambda|_L$ is exact and vanishes near $\p L$.
For admissible $L_0,L_1\subset W$ we define the \emph{Lagrangian Floer homology} $\HF(L_0,L_1)$ as follows.
Choose a Hamiltonian $H\colon W\to \R$ that is linear near the boundary with positive slope strictly less than the minimum of the lengths of Reeb chords from $\p L_0$ to $\p L_1$ and a family ${\bm J}=\{J_t\}_{t\in [0,1]}$ of compatible almost complex structures on $W$ that are cylindrical near the boundary.
The chain complex $\CF(L_0,L_1)$ is the $\Z_2$-vector space generated by the intersection points $L_0\cap \phi_H^{-1}(L_1)$, and the differential $\p^{\bm J}$ counts rigid Floer strips for $(H,{\bm J})$ with boundary on $L_0$ and $L_1$.
If $(H,{\bm J})$ is generic so that $L_0\cap \phi^{-1}_H(L_1)$ is transverse and that the moduli spaces defining the differential are cut out transversely, then $\HF(L_0,L_1)$ is well-defined and independent of the pair $(H,{\bm J})$.

Consider now the product $W\times W$ of a Liouville domain $(W,\lambda)$ which is a smooth manifold with corners.
Following \cite[Section~3.d]{Oa06}, we can round the corners to obtain a Liouville domain whose completion is unique up to Liouville isomorphism.
More precisely, choose a hypersurface $\Sigma\subset  W\times W$ such that
\begin{itemize}
	\item $\Sigma$ is contained in a small neighborhood of $(W\times \p W) \cup (\p W\times W)$; and
	\item $\Sigma$ is transverse to the Liouville vector field $Z=\pi_1^*X_\lambda+\pi_2^*X_\lambda$, where $\pi_i\colon W\times W\to W$ is the projection along the $i$th factor.
\end{itemize}
The closure of the component of $(W\times W)\setminus \Sigma$ that has no boundary is a Liouville domain $\widetilde{W\times W}$ equipped with 1-form $\lambda\oplus (-\lambda)$.
We call it a \emph{rounded product Liouville domain}.
\begin{remark}\label{rem: hypersurface_phi}
If $\phi\in \Symp(W,\lambda)$ is given, then we choose $\Sigma$ such that $\supp(\phi\times \phi)$ is contained in $\widetilde{W\times W}$ and is invariant under the involutions $\mathcal{R}_{\mathds{1}_W}(x,y)=(y,x)$ and $\mathcal{R}_\phi(x,y)=(\phi^{-1}(y),\phi(x))$.	
\end{remark}
\begin{lemma}\label{lem: graph_admissible}
For $\phi\in \Symp(W,\lambda)$ the graph Lagrangian
$$
\Gr(\phi)=\{(x,\phi(x))\in W\times W\mid x\in W\}\cap \widetilde{W\times W}
$$ is admissible in $\widetilde{W\times W}$.
In particular, the diagonal $\Delta=\Gr(\mathds{1}_W)\subset \widetilde{W\times W}$ is admissible.
\end{lemma}
\begin{proof}
Write $\phi^*\lambda=\lambda-df$ for some smooth function $f\colon W\to \R$.
We observe that
$$
\mathcal{R}_\phi^*\big(\lambda\oplus(-\lambda)\big) = (-\lambda+df)\oplus(\lambda+d(\phi^{-1})^*f) = \big((-\lambda)\oplus \lambda\big) + dF,
$$
where $F(x,y)=f(x)+f(\phi^{-1}(y))$ is a smooth function on $\widetilde{W\times W}$.
Since $\mathcal{R}_\phi^*F=F$, it follows that $dF|_{\Fix(\mathcal{R}_\phi)}=0$.
As $\Fix(\mathcal{R}_\phi)=\Gr(\phi)$ has a non-empty boundary in $\p( \widetilde{W\times W})$, which is indeed diffeomorphic to $\p W$, the lemma follows from \cite[Lemma~3.1]{KKL18}; while the lemma requires strict exactness, the proof remains valid in this setting.
\end{proof}
In the sequel we shall prove that $\HF(\phi)$ is canonically isomorphic to $\HF(\Delta,\Gr(\phi))$.
Since this is folklore in the closed case \cite[Example~3.8]{Se14_2}, we focus on how the argument adapts to the case of Liouville domains.
\begin{theorem}\label{thm: can_isom}
For every $\phi\in\Symp(W,\lambda)$ there exists a canonical isomorphism
$$
\HF(\phi)\cong \HF(\Delta,\Gr(\phi)).
$$
\end{theorem}
\begin{proof}
By the invariance properties, we can assume that $\phi$ is non-degenerate, equivalently, $\Delta$ and $\Gr(\phi)$ intersect transversely.
Since the generators of $\CF(\phi)$ and $\CF(\Delta,\Gr(\phi))$ are in bijection under $\Fix(\phi)\ni x\mapsto (x,x)\in \Delta\cap\Gr(\phi)$, it suffices to show that the differentials coincide.
For generic ${\bm J}=\{J_t\}_{t\in \R}\in\mathcal{J}_{\phi}$ we consider the family $\widetilde{{\bm J}}=\big\{(-J_{\frac{1-t}{2}})\oplus J_{\frac{1+t}{2}}\big\}_{t\in[0,1]}$ of split almost complex structures on $\widetilde{W\times W}$.
The moduli space of twisted Floer trajectories for ${\bm J}$ corresponds bijectively to that of Floer strips for $\widetilde{{\bm J}}$ via
$
u(s,t) \mapsto\widetilde{u}(s,t)=\left(u\big(\frac{s}{2},\frac{1-t}{2}\big),u\big(\frac{s}{2},\frac{1+t}{2}\big)\right)
$.
Moreover, if $u$ is regular, then  so is $\widetilde{u}$, where regularity means that the associated Fredholm operator is surjective.
A priori, $\widetilde{u}$ is contained in $W\times W$, but the argument below shows that it actually lies in $\widetilde{W\times W}$.
Using the assumption $\supp(\phi)\subset W\setminus \p W$ and the maximum principle, we can choose a compact region $K\subset W\setminus\p W$ that contains all fixed points of $\phi$ and all twisted Floer trajectories contributing to $\HF(\phi)$.
We assume that the hypersurface $\Sigma\subset W\times W$ is chosen to be disjoint from $K\times K$, while still satisfying all of our standing assumptions.
With this choice, the correspondence above ensures that every $\widetilde{u}$ lies in $\widetilde{W\times W}$, which finishes the proof.
\end{proof}

\section{Local-to-global identity for Floer-theoretic entropy}\label{sec: LtoG}
Consider a Liouville subdomain $V$ of a Liouville domain $(W,\lambda)$, that is, 
$V$ is a codimension~0 compact submanifold such that $(V,\lambda|_V)$ itself is a Liouville domain.
We abbreviate
$$
\Symp(W,\lambda;V)=\{\phi\in\Symp(W,\lambda)\mid \supp(\phi) \subset V\setminus \p V\}.
$$
For $\phi\in \Symp(W,\lambda;V)$ the restriction satisfies $\phi|_V\in \Symp(V,\lambda)$.
Conversely, any $\phi_V\in \Symp(V,\lambda)$ can be extended to $\phi\in \Symp(W,\lambda;V)$ by acting as the identity on the complement~$W\setminus \Int V$.

\begin{example} A typical example arises from an $A_k$-configuration $(L_1,\dots,L_k)$ of Lagrangian spheres.
Choose a compact neighborhood $V$ of $L_1\cup\cdots\cup L_k$ that is a Liouville subdomain of $(W,\lambda)$, and for each $i$ choose exact Dehn twists $\tau_i\colon V\to V$ along $L_i$ with $\supp (\tau_i)\subset V\setminus \p V$.
Then every product of the Dehn twists $\tau_1,\dots,\tau_k$ is contained in $\Symp(W,\lambda;V)$.
\end{example}

The main goal of this section is to prove the following core result.
This identity is easily obtained when considering topological entropy \cite[Proposition~3.1.7]{KaHa95}.
From the view of Floer theory, however, it is not obvious, since we need to relate $\HF(\phi)$ and $\HF(\phi|_V)$.
\begin{theorem}[Local-to-global identity]\label{thm: localtoglobal}
For every $\phi\in \Symp(W,\lambda;V)$ we have $$
h_{\floer}(\phi)=h_{\floer}(\phi|_V).
$$	
\end{theorem}
This theorem readily follows from the next result:
\begin{theorem}[Exact triangle]\label{thm: exact_triangle}
For every $\phi\in \Symp(W,\lambda;V)$ there exists an exact triangle
\begin{equation}\label{eq: exact_tri}
\begin{tikzcd}[row sep=1.5cm]
\HF(\phi|_V)  \arrow[r,"i"] & \HF(\phi) \arrow[d,"j"]\\
 & \arrow[ul] \Ho(W\setminus V)  
\end{tikzcd}
\end{equation}
where $\Ho(W\setminus V)=\bigoplus_{i=0}^{\dim W} \Ho_i(W\setminus V)$ denotes the total singular homology of $W\setminus V$.
\end{theorem}

\begin{remark}\label{rem: local_to_global}
	Theorem~\ref{thm: exact_triangle} also establishes the local-to-global identity for other types of growth;
	in particular, $\kappa(\phi)=\kappa(\phi|_V)$ holds.
Recall Remark~\ref{rem: linear_infinite} regarding the linear growth rate $\kappa(\phi)$.
	This result is specific to Liouville domains and may fail for closed symplectic manifolds, as shown in Example~\ref{ex: failure}.
\end{remark}
The proof of Theorem~\ref{thm: exact_triangle} is presented in Section~\ref{sec: proof_localtoglobal}.
We first use Theorem~\ref{thm: localtoglobal} to show that the Floer-theoretic entropy of a single Dehn twist on a Liouville domain vanishes.
\begin{example}\label{ex: singledehn}
Consider the disc cotangent bundle $D^*S^n$ equipped with the canonical Liouville form $\lambda_0$, and let $\tau\in \Symp(D^*S^n,\lambda_0)$ be a Dehn twist along the zero section $L$.
Since $\dim \HF_*(\tau^{2k})$ grows linearly as $k\to \infty$, see \cite[Proposition~4.7]{Ul19} (it computes linear growth of a fibered twist which is isotopic to $\tau^2$), it follows that $h_{\floer}(\tau^2)=0$.
We now show $h_{\floer}(\tau)=h_{\floer}(\tau^2)=0$.
\cite[Theorem~4.2]{Se01} shows that for $\phi\in\Symp(W,\lambda)$ there exists an exact triangle
\begin{equation}\label{eq: seidel_exact_tri}
\begin{tikzcd}[row sep=1.5cm]
\HF(\phi \tau)  \arrow[r] & \HF(\phi) \arrow[d]\\
 & \arrow[ul] \HF(\phi(L),L)  
\end{tikzcd}
\end{equation}
Here $\HF(\cdot,\cdot)$ is the Lagrangian Floer homology of exact Lagrangians that are disjoint from the boundary, see \cite[Section~3]{Se01_2}.
Putting $\phi=\tau^k$ together with the fact that $\tau(L)=L$, we obtain that for all $k\in \N$,
\begin{equation*}
\begin{split}
	\dim \HF(\tau^{k+1}) &+ \Lambda \ge \dim \HF(\tau^{k}) \\
	\dim \HF(\tau^{k}) &+ \Lambda \ge \dim \HF(\tau^{k+1}),
\end{split}	
\end{equation*}
where $\Lambda:=\dim \HF(L,L)$ is finite, yielding the inequalities 
\begin{equation*}
\begin{split}
	\dim \HF(\tau^{2k}) &+ k\Lambda \ge \dim \HF(\tau^{k}) \\
	\dim \HF(\tau^{k}) &+ k\Lambda \ge \dim \HF(\tau^{2k}).
\end{split}	
\end{equation*}
Therefore, $h_{\floer}(\tau)=h_{\floer}(\tau^2)$ as desired.
It follows from the local-to-global identity that every Dehn twist $\tau\in \Symp(W,\lambda)$ in a Liouville domain must have $h_{\floer}(\tau)=0$ since $D^*S^n$ exactly embeds into $W$.
\end{example} 
However, the local-to-global identity may fail if a Liouville domain is replaced by a closed symplectic manifold, as we now explain.

\begin{example}\label{ex: failure}
Let $\overline{\Delta}=\{(x,-x)\mid x\in S^2\}$ be the antidiagonal sphere of a monotone symplectic quadric surface $(M,\ow)=(S^2\times S^2,\ow_{S^2}\oplus\ow_{S^2})$, where $\ow_{S^2}$ denotes an area form on $S^2$.
Consider a Dehn twist $\tau\in \Symp(M,\ow)$ along $\overline{\Delta}$, supported in a neighborhood $U$ of~$\overline{\Delta}$, which is symplectomorphic to a neighborhood of the zero section in $T^*\overline{\Delta}$.
By a result of Gromov, $[\tau]\in \pi_0 \Symp(M,\ow)$ has finite order, specifically two.
This illustrates the failure of the local-to-global identity for linear growth; if the property  $\kappa(\tau)=\kappa(\tau|_{T^*\overline{\Delta}})$ held, then $[\tau]\in \pi_0\Symp(M,\ow)$ would have linear growth, as $\kappa(\tau_{T^*\overline{\Delta}})>0$ from Example~\ref{ex: singledehn}, and hence would have infinite order, see Remark~\ref{rem: linear_infinite}.
\end{example}

\subsection{Restrictions on twisted Floer trajectories}\label{sec: restrictions_traj}
To prove Theorem~\ref{thm: exact_triangle}, we utilize the results of \cite[Section~2.3]{CiOa18}, which impose restrictions on Floer trajectories traversing both $\Int V$ and $W\setminus \Int V$.
Specifically, we exclude the following three types of twisted Floer trajectories (see Figure~\ref{fig: rest_traj}):
\begin{itemize}
	\item[(I)] with both asymptotes in $V$, but intersecting $W\setminus \Int V$.
	\item[(II)] from a twisted orbit in $\Int V$ to a constant orbit in $W\setminus \Int V$.
	\item[(III)] with both asymptotes in $W\setminus \Int V$, but intersecting $\Int V$.
\end{itemize}
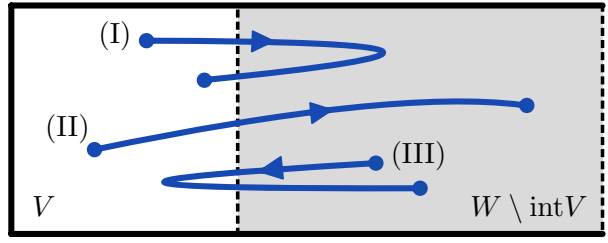
\begin{figure*}[t]
    \centering
\begin{overpic}[width=8cm]{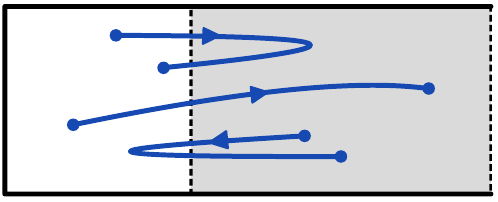}
\put(35, 75){\small (I)}
\put(15, 40){\small(II)}
\put(145, 30){\small(III)}
\put(10, 10){\small$V$}
\put(175, 10){\small$W \setminus \text{int}V$}
\end{overpic}
    \caption{The three types of trajectories to exclude.}
    \label{fig: rest_traj}
\end{figure*}
We start with the lemma asserting that any twisted Floer trajectory with asymptotic orbits in $V$ must stay in $V$, which excludes trajectories of Type (I) .
\begin{lemma}[No escape lemma]\label{lem: noescape}
Let $H\in \mathcal{H}_\phi$ and ${\bm J}\in \mathcal{J}_\phi$ such that $H=h(r)$ in cylindrical coordinates near $\p V$ and ${\bm J}$ is cylindrical near $\p V$.
If both asymptotes of a twisted Floer trajectory $u\colon \R\times \R\to W$ are contained in $V$, then $u$ is entirely contained in $V$.
\end{lemma}
\begin{proof}
The argument is identical to that of \cite[Lemma~2.2]{CiOa18}, with $W$ replacing $\widehat{W}$.
We argue by contradiction that $u$ leaves $V$.
After replacing $V$ by $\{r\le r_0\}$ for $r_0>1$ close to~1, we assume that $u$ leaves $V$ and is transverse to $\p V$.
Crucially, $u$ restricted to the compact smooth surface $S=u^{-1}(W\setminus \Int V)$ with boundary on $\p V$ is untwisted, since $\supp(\phi)\subset V$.
The rest of the argument, especially the energy estimate, relies only on the domain $W\setminus \Int V$.
We thus confirm that $u|_S$ is entirely contained in $\p V$, and the lemma follows without modification.
\end{proof}
Next, we rule out trajectories of Types (II) and (III) by adapting the argument of \cite[Lemma~2.5]{CiOa18} to the twisted case.
We explain the geometric setup and refer to \cite[Section~2.1]{CiOa18} for relevant definitions.

Consider a \emph{Liouville pair} $(W',V',\lambda')$, that is, $(V',\lambda'|_{V'})$ is a Liouville cobordism in a Liouville domain $(W',\lambda')$ such that 
$$
W'=W'_{\bottom}\circ V' \circ W'_{\topp},$$
where $W'_{\bottom}$ is a Liouville domain, $W'_{\topp}$ is a Liouville cobordism, and $\circ$ denotes the composition of Liouville cobordisms.
We write $\p^\pm V'$ for the negative and positive boundary of $V'$, respectively.
Additionally, suppose that $\phi' \in \Symp(W',\lambda';W'_{\bottom})$ is given.
Since $\supp (\phi')\subset W'_{\bottom}$, the map $\phi'$ is the identity on $V'\circ W'_{\topp}$.\\

\noindent Let $H'\colon \R\times W'\to \R$ be a Hamiltonian which has the following properties:
\begin{itemize}
	\item $H'\equiv c>0$ is constant on $V'$.
	\item $H'$ is strictly concave and strictly convex as a function of $r$ on $[1-\epsilon, 1]\times \p^-V'$ and $[1, 1+\epsilon]\times \p^+V'$, respectively.
	\item $H'$ is linear in $r$ on $(1-2\epsilon,1-\epsilon]\times \p^- V'$ and $(1+\epsilon,1+2\epsilon]\times \p^+ V'$, with positive slopes strictly less than $\Spec(\p^- V',\lambda'|_{\p^- V'})$ and $\Spec(\p^+ V',\lambda'|_{\p^+ V'})$, respectively.
\end{itemize}
Consider a Morse function $f\colon V'\to \R$ with the following properties:
\begin{itemize}
	\item $f$ is a function of $r$ near $\p V'$, and $\p^\pm V'$ are regular level sets; and
	\item The gradient of $f$ points inside and outside $V'$ along $\p^-V'$ and $\p^+V'$, respectively.
\end{itemize}
For $\epsilon>0$ small enough the \emph{$\epsilon$-thickening} of $V'$ inside $W'$ is denoted by
$$
V'_\epsilon=\big([1-\epsilon,1]\times \p^- V'\big)\cup V' \cup \big([1,1+\epsilon]\times \p^+ V'\big) \subset W'.
$$
Let $H_{f,\epsilon}'\colon \R\times W'\to \R$ be an admissible Hamiltonian such that
\begin{itemize}
	\item it is equal to $c+\epsilon^2f$ on $V'$ and to $H'$ outside $V'_\epsilon$; and
	\item it is strictly concave and strictly convex as a function of $r$ on $[1-\epsilon,1]\times \p^-V'$ and $[1,1+\epsilon]\times \p^+V'$, respectively.
\end{itemize}
Note that for $\epsilon>0$ small enough every twisted orbit of $H_{f,\epsilon}'$ contained in $V'$ is a constant orbit given by a critical point of $f$.\\

\noindent Let $\bm{J}'=\{J_t'\}_{t\in \R}$ be an admissible family of almost complex structures on $W'$ which are time-independent on $V'$, cylindrical near $\p V'$, and such that the gradient flow of $f$ is Morse--Smale with respect to the Riemannian metric $g_{\bm{J}'}=d\lambda'(\cdot, J'\cdot)$.

Now we establish the lemma that will exclude trajectories of Types (II) and (III).
\begin{lemma}\label{lem: constant_end}
For $\epsilon>0$ small enough the following holds.
\begin{enumerate}
	\item \label{item: 1} There is no twisted Floer trajectory for $(H_{f,\epsilon}',{\bm J}')$ from a twisted orbit in $W'_{\bottom}$ to a constant orbit in $V'$.
	\item \label{item: 2}  If both asymptotes of a twisted Floer trajectory for $(H_{f,\epsilon}',{\bm J}')$ are contained in $V'$, then it is entirely contained in $V'$.
\end{enumerate}
\end{lemma}
\begin{proof}
Following the proof of \cite[Lemma~2.5]{CiOa18}, we show \ref{item: 1}, and then the same arguments apply \ref{item: 2} as well.
Arguing by contradiction, assume that there exist a sequence $\epsilon_{\nu} \to 0$ and a sequence of twisted Floer trajectories $u_\nu\colon \R\times \R\to W'$ for $(H_{f,\epsilon}',\bm{J}')$ from $x_-$ to $p_+$, with $x_-$ an orbit of $H$ inside $W'_{\bottom}$ and $p_+$ a critical point of $f$.

Now, $V'$ can be seen as a Morse--Bott critical manifold for the action functional $\mathcal{A}_{H'}$, and $\mathcal{A}_{H_{f,\epsilon_\nu}'}$ is a sequence of Morse perturbations of $\mathcal{A}_{H'}$ along $V'$.
Since the critical manifold $V'$ is contained in the region on which $\phi'$ acts identically, the Morse--Bott compactness theorem in our situation can be proved, and the proof of \cite[Proposition~4.7]{BO_MB} carries almost verbatim.
See also \cite[Section~2]{DoSa94}.

Suppose that, possibly up to a subsequence, $u_\nu$ converges to a \emph{broken twisted Floer trajectory $[\bm{u}]$ with gradient fragments}, that is, in the definition \cite[Lemma~2.5]{CiOa18} and \cite[Definition~4.2]{BO_MB} of a broken Floer trajectory we replace (untwisted) Floer trajectories by twisted ones.
Since the action $\mathcal{A}_{H'}$ is constant on $V'$, the energy estimate for twisted Floer trajectories implies that each level of $[\bm{u}]$ contains at most one gradient trajectory for $f$.
In particular, this excludes the broken twisted Floer trajectories depicted in Figure~\ref{fig: mb_2}.
The limit $[\bm{u}]$ has a representative $\bar{\bm{u}}=(\bm{u}_1,\dots,\bm{u}_\ell)$ depicted as follows (see Figure~\ref{fig: mb_1}):
\begin{figure*}[t]
    \centering
\begin{overpic}[width=8cm]{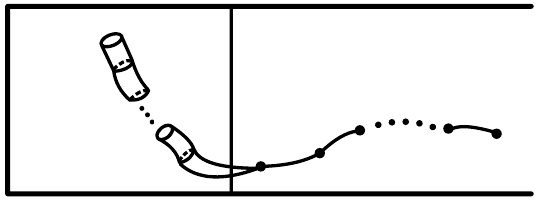}
\put(5, 70){\footnotesize $W'_{\text{bottom}}$ }
\put(100, 70){\footnotesize $V'$ }
\put(55, 65){\footnotesize $\boldsymbol{u}_1$}
\put(60, 53){\footnotesize $\boldsymbol{u}_2$}
\put(78, 30){\footnotesize $\boldsymbol{u}_{i-1}$}
\put(88, 20){\footnotesize $u_{i}$}
\put(118, 21){\footnotesize $\gamma_{i}$}
\put(130, 33){\footnotesize $\boldsymbol{u}_{i+1}$}
\put(195, 35){\footnotesize $\boldsymbol{u}_{\ell}$}
\put(210, 20){\footnotesize $p_+$}

\end{overpic}
    \caption{A broken twisted trajectory with gradient fragments.}
    \label{fig: mb_1}
\end{figure*}
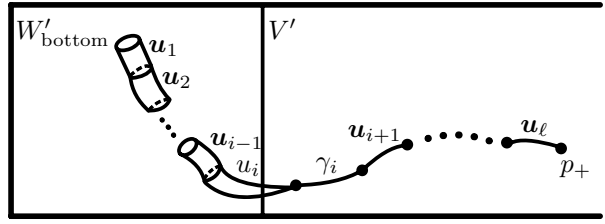
there exists $1\le i \le \ell$ such that
\begin{itemize}
	\item $\bm{u}_1,\dots,\bm{u}_{i-1}$ are twisted Floer trajectory for $H'$, with $\bm{u}_1(-\infty)=x_-$ and $\bm{u}_j(+\infty)=\bm{u}_{j+1}(-\infty)$ for all $1\le j\le i-2$.
	\item $\bm{u}_i=(u_i,\gamma_i)$ is a \emph{twisted Floer trajectory with one gradient fragment}, that is, $u_i$ is a twisted Floer trajectory for $H'$ and $\gamma_i\colon [0,+\infty)\to V'$ is a semi-infinite negative gradient trajectory for $f$, subject to the conditions: $\bm{u}_{i-1}(+\infty)=u_i(-\infty)$ if $i>1$ and $u_i(-\infty)=x_-$ if $i=1$; $u_i(+\infty)=\gamma_i(0)\in V'$; and $\gamma_i(+\infty)=p_+$ if $i=\ell$.
	\item $\bm{u}_{i+1},\dots,\bm{u}_\ell$ are negative gradient trajectories $\bm{u}_j=\gamma_j\colon \R\to V'$ for $f$, subject to the conditions: $\gamma_{j-1}(+\infty)=\gamma_j(-\infty)$ if $i+1\le j\le \ell$ and $\gamma_\ell(+\infty)=p_+$.
\end{itemize}
It follows from Lemma~\ref{lem: noescape} that twisted Floer trajectories $\bm{u}_1,\dots,\bm{u}_{i-1},u_i$ do not intersect $W'_{\topp}$.
The rest of the arguments are similar to those of \cite[Lemma~2.5]{CiOa18} since the part of a twisted Floer trajectory restricted to the cobordism $V'$ is untwisted.
To derive a contradiction, we examine the level $\bm{u}_i=(u_i,\gamma_i)$, which consists of a twisted Floer trajectory with one gradient fragment.\\\\
{\bf Case 1.} $\gamma_i(0)\in \Int V' = V'\setminus \p V'$.

Choose $s_0\gg0$ large enough such that the twisted Floer trajectory $u_i$ restricted to the set $[s_0,+\infty)\times \R$ is contained in $V'$ and is a $J'$-holomorphic curve $u\colon \dot{D}\cong [s_0,+\infty)\times S^1\to V'$, where we biholomorphically identify $[s_0,+\infty)\times S^1$ with a punctured disc $\dot{D}=D\setminus \{0\}$.
Since $0\in D$ is a removable singularity, $u$ (and hence $u_i$ as well) extends to a smooth map $u\colon D\to V'$.
Note that $S=u_i^{-1}(V')$ is a compact smooth surface with boundary (up to shifting the boundary $\p V'$ slightly in cylindrical coordinates) and that $u_i(\p S)\subset \p^- V'$.
Now we consider $u_i|_S$, namely the part of $u_i$ lying in $V'$, and the argument in the proof of Lemma~\ref{lem: noescape} yields that $u_i|_S$ is entirely contained in $\p^-V'$, which contradicts the assumption.\\\\
{\bf Case 2.} $\gamma_i(0)\in \p^+ V'$.

Pick $\delta>0$ such that $f$ has no critical point on $[1-\delta,1]\times \p^+V'$.
Since $[\bm{u}]$ is the limit of the sequence $u_\nu$, we can choose $\nu_0\ge 1$ such that the image of $u_{\nu_0}$ intersects $(1-\delta,1]\times \p^+V'$.
We observe that the part of $u_{\nu_0}$ lying in $[1-\delta,1]\times \p^+V'$ is untwisted and that both asymptotes $x_-, p_+$ of $u_{\nu_0}$ are located in $W'_{\bottom}\cup V'\, \setminus \, ([1-\delta,1]\times \p^+V')$.
Therefore, Lemma~\ref{lem: noescape} (with $\p V := \{1-\delta\}\times \p^+V'$) gives a contradiction.\\\\
{\bf Case 3.} $\gamma_i(0)\in \p^- V'$.

Since $\gamma_i(0)$ is not a critical point of $f$ and $\gamma_i$ is a negative gradient trajectory for $f$, it follows that $\gamma_i\colon [0,+\infty)\to V'$ must enter $V'$ in positive time.
However, the negative gradient $-\nabla f$ points outwards $V'$ along $\p^- V'$, which is a contradiction.\\\\
We now explain the proof of \ref{item: 2}.
Arguing by contradiction, suppose that there exist a sequence $\epsilon_\nu\to 0$ and a sequence of twisted Floer trajectories $u_\nu\colon \R\times \R\to W'$ for $(H_{f,\epsilon}',\bm{J}')$ leaving $V'$ whose asymptotes are critical points $p_\pm$ of $f$.
Since each $u_\nu$ does not intersect $W'_{\topp}$ by Lemma~\ref{lem: noescape}, it must intersect the interior of $W'_{\bottom}$.
In this case, the limit $[\bm{u}]$ of the sequence $u_\nu$ has a representative $\bar{\bm{u}}=(\bm{u}_1,\dots,\bm{u}_\ell)$, with $\ell\ge 1$, such that $\bm{u}_1,\dots,\bm{u}_\ell$ are negative gradient trajectories $\bm{u}_j=\gamma_j\colon \R \to V'$ for $f$, subject to the conditions: $\gamma_1(-\infty)=p_-$; $\gamma_{j-1}(+\infty)=\gamma_j(-\infty)$ if $2\le j\le \ell$; and $\gamma_\ell(+\infty)=p_+$.
Other types of limits are excluded, as illustrated in Figure~\ref{fig: mb_2}.
\begin{figure*}[t]
    \centering
\begin{overpic}[width=8cm]{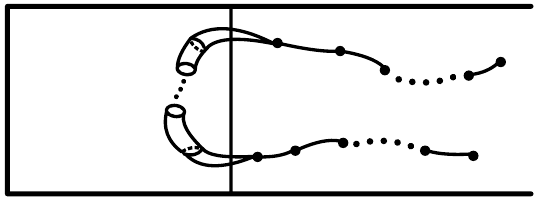}
\put(5, 40){\footnotesize $W'_{\text{bottom}}$}
\put(100, 40){\footnotesize $V'$}
\put(50, 25){\footnotesize $\boldsymbol{u}_{i-1}$}
\put(80, 8){\footnotesize $u_{i}$}
\put(55, 63){\footnotesize $\boldsymbol{u}_{j+1}$}
\put(77, 76){\footnotesize $u_{j}$}
\put(215, 50){\footnotesize $p_-$}
\put(205, 10){\footnotesize $p_+$}
\end{overpic}
    \caption{This broken twisted Floer trajectory cannot arise as the limit of the sequence~$u_{\nu}$.}
    \label{fig: mb_2}
\end{figure*}
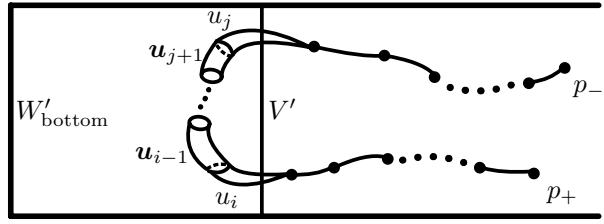
Observe that the image of each $\gamma_j$ is contained in the interior of $V'$ as $f$ has no critical points near $\p^-V'$.
Since the sequence $u_\nu$ converges to $[\bm{u}]$, the image of $u_{\nu}$ is arbitrarily $C^0$-close to that of $\bar{\bm{u}}$ as $\nu\to \infty$.
Therefore, for $\nu_0\gg1$ large enough the image of $u_{\nu_0}$ is contained in the interior of $V'$.
This contradicts the fact that $u_{\nu_0}$ intersects the interior of $W'_{\bottom}$.
\end{proof}

\subsection{Proof of Theorems~\ref{thm: localtoglobal} and \ref{thm: exact_triangle}}\label{sec: proof_localtoglobal}
Let us begin with the proof of Theorem~\ref{thm: exact_triangle}.
The strategy is to construct a short exact sequence of chain complexes associated to a suitable Hamiltonian $K\colon \R\times W\to \R$,
\begin{equation}\label{eq: exact_seq}
	0\to \CF(\phi|_V;K|_V) \overset{i}{\to} \CF(\phi;K) \overset{j}{\to} \CF(\phi;K)\, /\, \CF(\phi|_V;K|_V) \to 0
\end{equation}
such that the homologies of the first two complexes are $\HF(\phi|_V)$ and $\HF(\phi)$, respectively, while the quotient chain complex is identified with the Morse chain complex $\CM(K|_{W\setminus \Int V})$ for the smooth cobordism $W\setminus \Int V$.
To this end, we consider the following geometric setup as described in Section~\ref{sec: restrictions_traj}:
\begin{itemize}
	\item $(W',\lambda',\phi'):=(W,\lambda,\phi)$;
	\item $V':=W\setminus \Big(V \cup \big([1,1+\epsilon)\times \p V\big)\cup \big((1-2\epsilon,1]\times \p W\big)\Big)$; and
	\item $(K,{\bm J}):=(H'_{f,\epsilon},{\bm J}')$.
\end{itemize}
Then one can readily verify that
\begin{itemize}
	\item $W'_{\bottom}=V\cup \big([1,1+\epsilon]\times \p V\big)$ and $W'_{\topp}=[1-2\epsilon,1]\times \p W$;
 	\item $\phi'\in \Symp(W',\lambda';W'_{\bottom})$;
	\item $\p^- V'=\{1+\epsilon\}\times \p V$ and $\p^+ V'=\{1-2\epsilon\}\times \p W$;
	\item $(K,\bm{J})\in \mathcal{H}_\phi\times \mathcal{J}_\phi$ and $(K|_V,\bm{J}|_V)\in \mathcal{H}_{\phi|_V}\times \mathcal{J}_{\phi|_V}$; and
	\item $K$ is concave and convex as a function of $r$ on $[1,1+\epsilon]\times \p V$ and $[1-2\epsilon, 1-\epsilon]\times \p W$, respectively, and equal to $c+\epsilon^2f$ on $V'$.
\end{itemize}
See Figure~\ref{fig: CO_setting}.
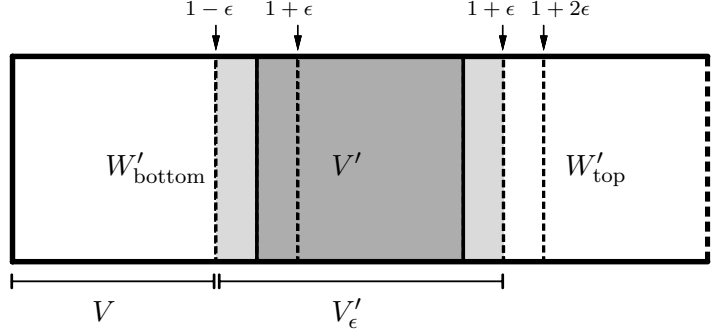
\begin{figure*}[t]
    \centering
\begin{overpic}[width=9.5cm]{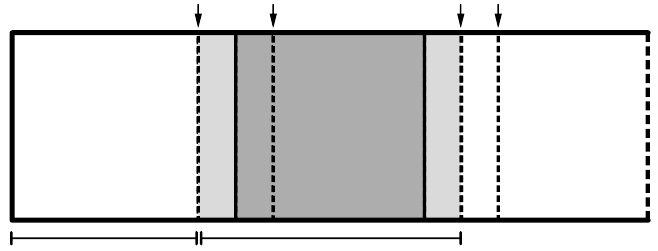}
\put(40, 45){\small$W'_{\text{bottom}}$}
\put(35, -10){\small$V$}
\put(125, 45){\small$V'$}
\put(213, 45){\small$W'_{\text{top}}$}
\put(125, -10){\small $V'_{\epsilon}$}
\put(70, 105){\scriptsize $1-\epsilon$}
\put(100, 105){\scriptsize $1+\epsilon$}
\put(176, 105){\scriptsize $1+\epsilon$}
\put(200, 105){\scriptsize $1+2\epsilon$}
\end{overpic}
    \caption{The geometric setup for $(W',V')$.}
    \label{fig: CO_setting}
\end{figure*}
Here the constants $c,\epsilon>0$ can be thought of as arbitrarily small.
Additionally, assume that ${\bm J}$ is generic so that the chain complexes $\CF(\phi|_V;K|_V)$ and $\CF(\phi;K)$ are well-defined.
We now explain the construction of \eqref{eq: exact_seq} as a consequence of Section~\ref{sec: restrictions_traj}.\\\\
{\bf Step 1.} \emph{The chain complex $(\CF(\phi|_V;K|_V),\p_V^{\bm{J}|_V})$ is a subcomplex of $(\CF(\phi;K), \p_W^{\bm{J}})$.}

Observe first that $\mathcal{P}(K)$ consists of two types of orbits:
\begin{enumerate}
	\item $\mathcal{P}(K|_V)$, twisted orbits  contained in $V$; and
	\item $\mathcal{P}(K|_{W\setminus \Int V})$, twisted orbits contained in $W \setminus \Int V$, which are constant orbits given by critical points of $f$.
\end{enumerate}
Let $i\colon \CF(\phi|_V;K|_V)\to \CF(\phi;K)$ denote the inclusion induced by $\mathcal{P}(K|_V)\subset \mathcal{P}(K)$.
By part~\eqref{item: 1} of Lemma~\ref{lem: constant_end}, the differential $\p_W^{\bm{J}}$ restricts to $\CF(\phi|_V;K|_V)$, and this restriction coincides with $\p^{\bm{J|_V}}_V$ by Lemma~\ref{lem: noescape}.
Therefore, $i$ is a chain map, which proves the assertion.
\\\\
{\bf Step 2.} \emph{The chain complex $(\CF(\phi;K)\, /\, \CF(\phi|_V;K|_V),\overline{\p^{\bm J}_W})$ is isomorphic to the Morse chain complex $(\CM(K|_{W\setminus \Int V}),\p_{\morse})$, where $\p_{\morse}$ denotes the differential for the Morse--Smale pair $(K|_{W\setminus \Int V},g_{\bm{J}}|_{W\setminus \Int V})$ on the smooth cobordism $W\setminus \Int V$.}

Note that there is a correspondence between the generators of $\CF(\phi;K)\, /\, \CF(\phi|_V;K|_V)$ and $\CM(K|_{W\setminus \Int V})$, as both are generated by the orbits of $\mathcal{P}(K|_{W\setminus \Int})$.
This correspondence provides an identification $\CF(\phi;K)\, /\, \CF(\phi|_V;K|_V) \cong \CM(K|_{W\setminus \Int V})$ as vector spaces.
By part \eqref{item: 2} of Lemma~\ref{lem: constant_end}, every twisted Floer trajectory for $K$ with both asymptotes in $W\setminus \Int V$, which are necessarily contained in $V'$, is entirely contained in $W\setminus \Int V$.
Moreover, it is untwisted since $\phi$ acts identically on the cobordism.
Since $K|_{W\setminus \Int V}$ can be chosen to be $C^2$-small so that every Floer trajectory contained in $W\setminus \Int V$ is a Morse trajectory for $K|_{W\setminus \Int V}$, see \cite[Theorem~7.3]{SaZe92}, 
it follows that the two differentials $\overline{\p^{\bm{J}}_W}\equiv \p_{\morse}$ on $\CF(\phi;K)\, /\, \CF(\phi|_V;K|_V)$ and $\CM(K|_{W\setminus \Int V})$ coincide up to identification.\\\\
{\bf Step 3.} \emph{Proof of Theorems~\ref{thm: localtoglobal} and \ref{thm: exact_triangle}}

From Steps 1 and 2, we have the short exact sequence \eqref{eq: exact_seq}.
In view of the isomorphisms $\HM(K|_{W\setminus \Int V},g_{\bm{J}}|_{W\setminus \Int V})\cong \Ho(W\setminus \Int V)\cong \Ho(W\setminus V)$, where the first homology denotes the (total) Morse homology, the associated long exact sequence is equal to \eqref{eq: exact_tri}, as sought.
Substituting $\phi$ by $\phi^k$ in \eqref{eq: exact_tri} for $k\ge 1$ yields the inequalities
\begin{equation}\label{eq: rankest}
	\begin{split}
	\dim \HF(\phi|_V^k) + \dim \Ho(W\setminus V) &\ge \dim \HF(\phi^k),\\
	\dim \HF(\phi^k) + \dim \Ho(W\setminus V) &\ge \dim \HF(\phi|_V^k).	
\end{split}
\end{equation}
Since $\dim \Ho(W\setminus V)$ is finite and independent of $k$, it follows from the definition of Floer-theoretic entropy that $h_{\floer}(\phi) = h_{\floer}(\phi|_V)$.

\section{Smith-type inequality for fixed point Floer homology}\label{sec: smith_ineq}
This section establishes a Smith-type inequality for fixed point Floer homology, see Theorem~\ref{thm: smith_floer}.
This result follows as a direct application of the corresponding inequality for Lagrangian Floer homology, proved by Seidel--Smith \cite{SS}.
We begin by recalling the formulation of their result relevant to our context.

An \emph{equivariant} tuple $(W,\lambda,L_0,L_1,\sigma)$ consists of a Liouville domain $(W,\lambda)$, an \emph{exact symplectic involution} $\sigma$ of $W$, that is, $\sigma^*\lambda=\lambda$ and $\sigma^2=\mathds{1}$, and $\sigma$-invariant admissible Lagrangians $L_0,L_1\subset W$, such that the fixed locus $S=\Fix(\sigma)$ is non-empty and connected.
Note that $(S,\lambda|_S)$ is a Liouville domain with $\p S= \p W\cap S$.
Given an equivariant tuple $(W,\lambda,L_0,L_1,\sigma)$, let $n$ and $n_{\anti}$ be the complex dimensions of $W$ and $S$, respectively.
Here we introduce the following notation and \cite[Definition 18]{SS}:
\begin{itemize}
	\item $L^{\inv}_i:=L_i\cap S$ denotes the invariant part of $L_i$.
	\item $TW^{\anti}$ denotes the pullback bundle of the normal bundle $N_WS$ via the projection $[0,1]\times S\to S$.
	\item $TL_i^{\anti}:=N_{L_i}L_i^{\inv}$ denotes the normal bundle of $L_i^{\inv}$ in $L_i$.
\end{itemize}

\begin{definition}
	A \emph{stable normal trivialization} for an equivariant tuple $(W,\lambda,L_0,L_1,\sigma)$ consists of the following data:
\begin{itemize}
	\item A unitary trivialization of vector bundles over $S$
$$
\phi\colon TW^{\anti}\oplus \C^N \to \C^{n_{\anti}+N} \qquad \text{for some $N\ge 0$}.
$$
	\item Lagrangian subbundles $\Lambda_i\subset (TW^{\anti}\oplus \C^N)|_{[0,1]\times L_i^{\inv}}$ for $i=0,1$, such that 
\begin{align*}
\Lambda_0|_{\{0\}\times L^{\inv}_0}=TL_0^{\anti}\oplus \R^N &\quad\text{and}\quad
\phi(\Lambda_0|_{\{0\}\times L_0})=\R^{n_{\anti}+N} \\
\Lambda_1|_{\{1\}\times L^{\inv}_1}=TL_1^{\anti}\oplus i\R^N &\quad\text{and}\quad
\phi(\Lambda_1|_{\{1\}\times L_1})=i\R^{n_{\anti}+N} 
\end{align*}

\end{itemize}
If this exists, then $(W,\lambda,L_0,L_1,\sigma)$ is called \emph{stably normally trivialisable}.
\end{definition}
The Smith-type inequality of Seidel--Smith is as follows:
\begin{theorem}[{\cite[Theorem~1]{SS}}]\label{thm: seidel_smith}
Let $(W,\lambda,L_0,L_1,\sigma)$ be an equivariant tuple that is stably normally trivialisable.
Then, we have the inequality: 
$\dim \HF(L_0,L_1)\ge \dim \HF(L_0^{\inv},L_1^{\inv})$.
\end{theorem}
Note that if $\phi\in\Symp(W,\lambda)$ commutes with $\sigma$, then it restricts to $\phi|_S\in \Symp(S,\lambda|_S)$.
We now state the main result of this section:
\begin{theorem}\label{thm: smith_floer}
Let $(W,\lambda)$ be a Liouville domain, and let $\sigma$ be an exact symplectic involution of $W$ whose fixed locus $S$ is non-empty and connected.
Suppose that the normal bundle $N_WS$ is unitarily trivialisable and $\Ho^i(S)$ is free for all $i\ge 1$.
Then, for every $\phi\in \Symp(W,\lambda)$ that commutes with $\sigma$, we have the inequality:
$$
\dim \HF(\phi) \ge \dim \HF(\phi|_S).
$$
\end{theorem}

\begin{proof}
We begin by proving the following claim.
Although this is slightly more than we need, we prove it as it is of independent interest and does not take much effort.
\\\\
{\bf Claim.} \emph{For all $\phi_0,\phi_1\in \Symp(W,\lambda)$ that commute with $\sigma$, the tuple
\begin{equation}\label{eq: tuple}
	(\widetilde{W\times W},\lambda\oplus(-\lambda),\Gr(\phi_0),\Gr(\phi_1),\sigma\times \sigma)	
\end{equation}
is equivariant and stably normally trivialisable.}
	
We first observe that \eqref{eq: tuple} is equivariant.
Indeed, $\Gr(\phi_0)$ and $\Gr(\phi_1)$ are admissible by Lemma~\ref{lem: graph_admissible} and $\sigma$-invariant as $\phi$ commutes with $\sigma$.
The fixed locus $\widetilde{S\times S}$ of $\sigma\times \sigma$ is non-empty and connected as $S$ is so.
To prove that \eqref{eq: tuple} is stably normally trivialisable, we follow an elegant argument by Hendricks \cite[Section~7]{He12}, where the proof splits into several steps.
For notational simplicity, we write $M=\widetilde{W\times W}$ and $L_i=\Gr(\phi_i)$ for $i=0,1$.
\\\\
{\bf Step 1.} \emph{$TM^{\anti}=T(\widetilde{W\times W})^{\anti}$ is trivial as a complex vector bundle.}

This follows immediately from the definition of $TM^{\anti}$ and the assumption that $N_WS$ is unitarily trivialisable.\\\\
{\bf Step 2.} \emph{$TL_i^{\anti}$ is trivial as a real vector bundle for $i=0,1$.}

For $i=0,1$, write $\phi=\phi_i$ and $L=\Gr(\phi)$.
Then $L^{\inv}=\Gr(\phi|_S)$ and $TL^{\anti}:=N_LL^{\inv}$ are defined accordingly.
In view of the description
$$
TL^{\anti}=N_{\Gr(\phi)}\Gr(\phi|_S)=\{(v,\phi_*v) \mid v\in N_WS\},
$$
the canonical bundle map
$$
N_WS \to TL^{\anti}, \quad v\mapsto (v,\phi_*v)
$$
is an isomorphism, implying that $TL^{\anti}$ is trivial as $N_WS$ is.
This completes Step 2.\\

From now on, we carefully follow the $K$-theoretic arguments in the proof of \cite[Theorem~3.11]{He12}.
We refer to it for relevant details and definitions.
Write $X:= (L_0^{\inv}\times \{0\})\cup (L_1^{\inv}\times \{1\})\subset M^{\inv}\times [0,1]$, and let $J$ be an almost complex structure on $M^{\inv}=S\times S$.
Choose a complex trivialization of $TM^{\anti}|_X$ whose restriction to $(TL_0^{\anti}\times \{0\}) \cup (JTL_1^{\anti}\times \{1\})$ is a preferred real trivialization (induced by trivializations of $TL_i^{\anti}$) and we want to show the relative vector bundle $[(TM^{\anti})_{\rel}]\in \widetilde{K}^0((M^{\inv}\times [0,1],X)\cong \widetilde{K}^0((M^{\inv}\times [0,1])/X)$ is trivial, where $\widetilde{K}^0(\cdot)$ denotes the $K$-theory given in \cite[Section~6]{He12}.
Then the trivialization extends to that of $TM^{\anti}$, and the claim follows from \cite[Proposition~7.1]{He12}.\\\\
{\bf Step 3.} \emph{$[(TM^{\anti})_{\rel}]\in \widetilde{K}^0((M^{\inv}\times [0,1])/X)$ is a torsion element.}

Since the reduced Chern character
$$
\widetilde{\ch}\colon \widetilde{K}^0((M^{\inv}\times [0,1])/X)\otimes \Q \to \widetilde{\Ho}^*((M^{\inv}\times [0,1])/X;\Q)
$$
is an isomorphism, it is enough to show that the Chern classes of $(TM^{\anti})_{\rel}$ are trivial.
We first claim that the map $j^*\colon \Ho^i(M^{\inv}\times [0,1])\to \Ho^i(X)$ induced by the inclusion is surjective for all $i\ge 1$.
Note that the inclusion is given by
$$
j\colon (\Gr(\phi_0|_S)\times \{0\})\cup (\Gr(\phi_1|_S)\times \{1\}) \hookrightarrow (\widetilde{S\times S})\times [0,1].
$$
Since $\Gr(\phi_0|_S)$ and $\Gr(\phi_1|_S)$ are diffeomorphic to $S$, and they are graphs in the product space $\widetilde{S\times S}$, the map $j^*$ is identified with
$$
\Ho^i(S)\otimes \Ho^i(S) \to \Ho^i(S\sqcup S)\cong \Ho^i(S)\oplus \Ho^i(S),\quad \alpha\otimes 1 + 1\otimes \beta \mapsto \alpha + \beta,
$$
yielding that $j^*$ is surjective in degrees $i\ge 1$.
Let $q\colon M^{\inv}\to (M^{\inv}\times [0,1])/X$ denote the quotient map.
Since the Chern class $c_i(TM^{\anti})$ vanishes for all $i\ge 1$ and $TM^{\anti}=q^*((TM^{\anti})_{\rel})$, we have $q^*c_i((TM^{\anti})_{\rel})=0$.
From a relative long exact sequence
$$
\to \widetilde{\Ho}^m((M^{\inv}\times [0,1])/X)\to \widetilde{\Ho}^m(M^{\inv})\to \widetilde{\Ho}^m(X)\to \widetilde{\Ho}^{m+1}((M^{\inv}\times [0,1])/X)\to,
$$
the map $q^*\colon \widetilde{\Ho}^m((M^{\inv}\times [0,1])/X)\to \widetilde{\Ho}^m(M^{\inv})$ is injective for $m\ge 2$.
We thus conclude that $c_i((TM^{\anti})_{\rel})=0$ for all $i$, as asserted.\\\\
{\bf Step 4.} \emph{$\widetilde{K}^0((M^{\inv}\times [0,1])/X)$ is torsion-free.}

As a result of Atiyah--Hirzebruch \cite[Proposition~6.10]{He12}, it suffices to show that $\Ho^i((M^{\inv}\times [0,1])/X)$ is torsion-free and finitely generated for all $i\ge 0$.
Since $q^*$ is injective for $m\ge 2$ and $\widetilde{\Ho}^*(M^{\inv})\cong \widetilde{\Ho}^*(S)\otimes \widetilde{\Ho}^*(S)$ is free abelian by the assumption on $S$, the group $\widetilde{\Ho}^m((M^{\inv}\times [0,1])/X)$ is torsion-free and finitely generated for $m\ge 2$.
To deal with the case $m=1$, we investigate the long exact sequence
\begin{align*}
0=\widetilde{\Ho}^0(M^{\inv}\times [0,1]) &\to \widetilde{\Ho}^0(X) \cong \widetilde{\Ho}^0(S\sqcup S)\cong \Z \\
& \to \widetilde{\Ho}^{1}((M^{\inv}\times [0,1])/X)\to \widetilde{\Ho}^{1}(M^{\inv}\times [0,1]) \to.
\end{align*}
Since $\widetilde{\Ho}^{1}(M^{\inv}\times [0,1])$ is free abelian, so is the image of $q^*$.
Now, in view of the short exact sequence
$$
0 \to \Z \to  \widetilde{\Ho}^{1}((M^{\inv}\times [0,1])/X)  \to \im q^* \to 0,
$$
the first homology group $\widetilde{\Ho}^{1}((M^{\inv}\times [0,1])/X)$ is torsion-free and finitely generated.\\

Combining Steps 3 and 4, we deduce that $[(TM^{\anti})_{\rel}]\in \widetilde{K}^0((M^{\inv}\times [0,1])/X)$ is trivial, which completes the proof of the claim.\\\\
Now let $\phi\in\Symp(W,\lambda)$ commute with $\sigma$.
We assume that the hypersurface $\Sigma$ introduced in Section~\ref{sec: lagrfloer} is invariant under $\sigma\times\sigma$.
Consider admissible Lagrangians $\Delta_S=\Gr(\mathds{1}_S)$ and $\Gr(\phi|_S)$ in $\widetilde{S\times S}$, given by the invariant parts of $\Delta=\Gr(\mathds{1}_W)$ and $\Gr(\phi)$, respectively.
In view of the isomorphisms from Theorem~\ref{thm: can_isom},
$$
\HF(\phi)\cong \HF(\Delta,\Gr(\phi))\quad \text{and} \quad \HF(\phi|_S)\cong \HF(\Delta_S,\Gr(\phi|_S)),
$$
it suffices to show the inequality
$$
	\dim \HF(\Delta,\Gr(\phi))\ge \dim \HF(\Delta_S,\Gr(\phi|_S)).
$$
This follows from Theorem~\ref{thm: seidel_smith} together with the claim above ($\phi_0=\mathds{1}$ and $\phi_1=\phi$).
\end{proof}

\begin{remark}
In Theorem~\ref{thm: smith_floer} the case where the fixed locus $S$ is not connected also holds with a generalization of stable normal trivializations, see \cite[Section~3.5]{SS}.
\end{remark}


\section{Milnor fibres of the $A_k$-singularity, and all that}\label{sec: milnorfibre}
In this section, we recollect the geometry of Milnor fibres of $A_k$-singularities with symplectic involutions.
This serves as a local model for a compact neighborhood of an $A_k$-configuration of Lagrangian spheres in a Liouville domain, where we will employ Theorem~\ref{thm: smith_floer}.
Subsequently, we discuss Penner-type diffeomorphisms on compact oriented surfaces, which provide explicit pseudo-Anosov maps, and their generalization to Liouville domains.
\subsection{Dehn twists via $A_k$-configurations in Milnor fibres}\label{sec: milnorfibres}
We follow the descriptions given in \cite[Section~6c]{KhSe02}.
For a positive integer  $k\ge 1$ and $\epsilon>0$ small enough, we define the \emph{Milnor fibre of the $A_k$-singularity} as
$$
V=\left\{\boldsymbol{z}=(z_0,\dots,z_n)\in \C^{n+1} \mid   z_0^2+\dots+z_{n-1}^2+z_n^{k+1}=\epsilon,\ |\boldsymbol{z}|\le 1 \right\}
$$
endowed with the standard Liouville form $\lambda_{\st}=\frac{i}{4}\sum_j(z_jd\bar{z}_j-\bar{z}_jdz_j)$.
Then $(V,\lambda_{\st})$ is a Liouville domain of dimension $2n$ equipped with the exact symplectic involution
\begin{equation}\label{eq: involution}
\sigma(z_0,\dots,z_n) = (-z_0,\dots,-z_{n-2},z_{n-1},z_n).	
\end{equation}
This involution also appears in \cite[Lemma~6.14]{KhSe02}.
The fixed locus of $\sigma$ is the double branched cover of a disc at $(k+1)$ points, namely
$$
S:=\Fix(\sigma)=\left\{\boldsymbol{z}=(0,\dots,0,z_{n-1},z_n)\in \C^{n+1} \mid z_{n-1}^2 + z_n^{k+1} = \epsilon,\ |\boldsymbol{z}|\le 1 \right\},
$$
and it is a compact Riemann surface of genus $g$ with $m$ boundary circles, where $m,g$ are determined as follows: $\displaystyle(m,g)=\Big(1,\frac{k}{2}\Big)$ if $k$ is even, and $\displaystyle(m,g)=\Big(2,\frac{k-1}{2}\Big)$ if $k$ is odd. 
Note that $(S,\lambda_{\st}|_S)$ is a Liouville domain with $\p S=\p V\cap S$.

\begin{lemma}\label{lem: configuration} There is an $A_k$-configuration of $\sigma$-invariant Lagrangian spheres $L_1,\dots,L_k$ in~$V$, which descends to an $A_k$-configuration of Lagrangian circles $\ell_i:=L_i\cap S$ in $S$.
\end{lemma}
\begin{proof}
The construction is fairly well-known: the Lagrangian spheres arise from vanishing cycles of a Lefschetz fibration, see \cite[Section~6c]{KhSe02}.
Take the disc $D=\{z\in \C \mid |z^{k+1}-\epsilon|^2+|z|^2\le 1 \}$ and the set of $(k+1)$ points $\Delta=\{z\in \C\mid z^{k+1}=\epsilon\}$.
Consider the Lefschetz fibration
$$
\pi\colon V \to D,\quad \pi(z_0,\dots,z_n)=z_n
$$
whose regular fibre is symplectomorphic to a disc cotangent bundle of $S^{n-1}$.
Given a smooth simple path $c$ in $D$ intersecting $\Delta$ only at its endpoints, we can  parallel transport the subset corresponding to the zero section
$$
\Sigma_z=\left\{(z_0,\dots,z_{n-1},z)\in V \,\Big|\, 
\begin{array}{l}
|z_0|^2+\dots+|z_{n-1}|^2=|z^{k+1}-\epsilon|, \\
z_i\in \sqrt{\epsilon-z^{k+1}}\R \text{ for $i=0,\dots,{n-1}$}
\end{array}
\right\}
$$
of the fibre $Q_z=\{\boldsymbol{z}\in V \mid  z_0^2+\dots+z_{n-1}^2+z^{k+1}=\epsilon\}=\pi^{-1}(z)$ at a regular value $z\in c$,
so that we obtain a Lagrangian sphere $L_c=\bigcup_{z\in c} \Sigma_z$ in $V$, see  \cite[(a) of Lemma~6.12]{KhSe02}.
Since $\Sigma_z$ is $\sigma$-invariant for all $z\in D$, so is $L_{c}$.
Taking minimally intersecting curves $c_1,\dots,c_k$ in~$D$ corresponding to an $A_k$-configuration as in Figure~\ref{fig: minimal_curves},
\begin{figure*}[t]
    \centering
\begin{overpic}[width=5cm]{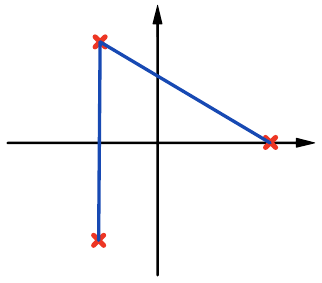}
\put(25, 70){$c_2$}
\put(88, 90){$c_1$}
\end{overpic}
\vspace{-0.5cm}
    \caption{Minimal curves $c_1$ and $c_2$ for the case $k=2$.}
    \label{fig: minimal_curves}
\end{figure*}
we obtain an $A_k$-configuration of Lagrangian spheres $L_{c_1},\dots,L_{c_k}$ having the desired properties.
Specifically, since $\Fix(\sigma|_{\Sigma_z})$ is a circle or a single point, depending on whether $z\in c$ is a regular value of~$\pi$, the Lagrangian circles $\ell_i=\bigcup_{z\in c}\Fix(\sigma|_{\Sigma_z})$ form the same $A_k$-configuration as before.
\end{proof}
\begin{remark}\label{rem: inv_ak_spheres}
In a similar vein, the Lagrangian spheres in Lemma~\ref{lem: configuration} are invariant under the exact symplectic involutions $\sigma_j(z_0,\dots,z_n)=(-z_0,\dots,-z_{n-j-1},z_{n-j},\dots ,z_n)$ for all $j=1,\dots,n$, where $\sigma_1=\sigma$ and $\sigma_n=\mathds{1}_V$.
\end{remark}
We discuss the geometric setup for the Dehn twists that we consider.
Let $L_1,\dots,L_k\subset V$ be an $A_k$-configuration of Lagrangian spheres as in Lemma~\ref{lem: configuration}.
For each $i=1,\dots,k$, we can choose a Dehn twist $\tau_{L_i}\in \Symp(V,\lambda_{\st})$ along $L_i$ such that $\tau_{L_i}$ commutes with $\sigma_j$ for all $j=1,\dots,n$, and $\tau_{L_i}|_S\in \Symp(S,\lambda_{\st}|_S)$ coincides with a Dehn twist $\tau_{\ell_i}$ along $\ell_i$, that is, $\tau_{L_i}|_S=\tau_{\ell_i}$ for some parametrisation of $\ell_i$, see the proof of \cite[Lemma~4.9]{To15}.
The above construction of the Dehn twists relies on the commutativity of $\sigma_j$ and $\sigma_i$.
We abbreviate
$$
\Symp(V,\lambda_{\st};{\bm \sigma})=\{\phi\in \Symp(V,\lambda_{\st}) \mid \text{$\phi$ commutes with $\sigma_j$ for all $j=1,\dots,n$} \}.
$$
The following result is immediate from Lemma~\ref{lem: configuration} and the setup above:
\begin{lemma}\label{lem: lift_dehntwists}
Any product $\phi=\tau_{L_{i_1}}^{(-1)}\circ \dots \circ \tau_{L_{i_m}}^{(-1)}\in \Symp(V,\lambda_{\st};{\bm \sigma})$ of (positive or negative) Dehn twists along~$L_{i_r}$ restricts to the product of Dehn twists along $\ell_{i_r}$,
$$
\phi|_S=\tau_{\ell_{i_1}}^{(-1)}\circ \dots \circ \tau_{\ell_{i_m}}^{(-1)}\in \Symp(S,\lambda_{\st}|_S).
$$
Conversely, any product $\phi_S\in \Symp(S,\lambda_{\st}|_S)$ of (positive or negative) Dehn twists along~$\ell_{i_r}$ lifts to the product $\phi\in \Symp(V,\lambda_{\st};{\bm \sigma})$ of the corresponding Dehn twists along~$L_{i_r}$.
\end{lemma}
By virtue of the Smith-type inequality, we obtain the main result of this section:
\begin{theorem}\label{thm: milnor_smith}
For every $\phi\in \Symp(V,\lambda_{\st};{\bm \sigma})$ we have $h_{\floer}(\phi)\ge h_{\floer}(\phi|_S)$.
\end{theorem}
\begin{proof}
This readily follows from applying Theorem~\ref{thm: smith_floer} inductively.
For $j=1,\dots,n$, let
$$
V_j=\left\{\boldsymbol{z}=(0,\dots,0,z_{n-j},\dots,z_n)\in \C^{n+1} \,\Big|\,   z_{n-j}^2+\dots+z_{n-1}^2+z_n^{k+1}=\epsilon,\ |\boldsymbol{z}|\le 1 \right\}\subset V
$$
be the Milnor fibre of the $A_k$-singularity of dimension $2j$.
Note that $\Fix(\sigma_j|_{V_j})=V_{j-1}$, $V_1=S$, and $V_n=V$.
Note that the first Chern class $c_1(V_j)$ vanishes, and hence $c_1(N_{V_j}V_{j-1})=0$.
This implies that the normal bundle $N_{V_j}V_{j-1}$ is unitarily trivial, since it is a complex line bundle.
Moreover, $\Ho^i(V_{j-1})$ is free for all $i\ge 1$ as the Milnor fibre $V_{j-1}$ is homotopy equivalent to the wedge sum of spheres.
By Theorem~\ref{thm: smith_floer}, $h_{\floer}(\phi|_{V_j})\ge h_{\floer}(\phi|_{V_{j-1}})$ for all $j=2,\dots,n$, yielding the assertion.
\end{proof}

\subsection{Symplectomorphisms of Penner-type }\label{sec: penner_milnor}
Henceforth, we are interested in finding $\phi_S\in \pi_0\Symp(S,\lambda_{\st}|_S)$, which is a product of Dehn twists from an $A_k$-configuration and satisfies $h_{\floer}(\phi_S)>0$.
By the classification result of Thurston \cite{Th88}, there is a trichotomy for mapping classes of diffeomorphisms on a compact oriented surface $\Sigma$, that is, every $[\phi]\in\pi_0\Diff(\Sigma)$ belongs to one of the following types: \emph{periodic}, \emph{reducible}, or \emph{pseudo-Anosov}.
The reader is referred to a comprehensive exposition \cite{FaMa12} for further details.
We focus only on the latter type, motivated by the following result of Cotton-Clay:
\begin{theorem}[{\cite[Corollary~1.7]{Co09}}]\label{thm: cotton_clay}
Let $\Sigma$ be a compact symplectic surface with boundary, and let $\lambda$ be a Liouville form on $\Sigma$.
Then $[\phi]\in\pi_0\Symp(\Sigma,\lambda)$ is pseudo-Anosov if and only if $h_{\floer}(\phi)>0$.
\end{theorem}
Inspired by the celebrated work of Penner, we construct a pseudo-Anosov map $\phi_S\in \Symp(S,\lambda_{\st}|_S)$ as a product  of Dehn twists from an $A_k$-configuration.
This explicit construction allows us to lift it to obtain a map in $\Symp(V,\lambda_{\st};{\bm \sigma})$ via Lemma~\ref{lem: lift_dehntwists}.
Note that not all pseudo-Anosov classes arise from Penner's construction, see \cite[Section~1]{ShSt15}.

We shall recall the relevant notions from \cite{Pe88} and \cite[Section~6]{Co09}.
A \emph{multicurve} in a compact oriented surface $\Sigma$ with boundary is a collection of disjoint embedded (oriented) loops in~$\Sigma$.
We say that a collection of embedded loops in $\Sigma$ \emph{fills} $\Sigma$ if its complement in~$\Sigma$ is a disjoint union of (possibly punctured) discs with more than two edges.

\begin{theorem}[Penner's construction {\cite[Theorem~3.1]{Pe88}}]\label{thm: penner_construction}
Let $\mathcal{C}=\{c_i\}$ and $\mathcal{D}=\{d_j\}$ be multicurves in a compact oriented surface $\Sigma$ with boundary such that $\mathcal{C}\cup\mathcal{D}$ fills $\Sigma$.
If $\phi$ is a product of positive Dehn twists along the $c_i$'s and negative Dehn twists along the $d_j$'s such that all loops in $\mathcal{C}\cup\mathcal{D}$ appear at least once as the generating loop of a Dehn twist, then $\phi$ is pseudo-Anosov.
\end{theorem}
In this case, we say that $\phi$ is given by \emph{Penner's construction} with respect to the multicurves $\mathcal{C}$ and $\mathcal{D}$.
We emphasize that the order of two multicurves $\mathcal{C}$ and $\mathcal{D}$ cannot be discarded, as they correspond to \emph{positive} and \emph{negative} Dehn twists, respectively.

\begin{example}[$A_2$-singularity]\label{ex: pennerA2}
Let $V$ be the Milnor fibre of the $A_2$-singularity of dimension $2n\ge 4$.
Recall from Section~\ref{sec: milnorfibres} that the two-dimensional Milnor fibre of the $A_2$-singularity is a compact Riemann surface $S$ of genus one with one boundary circle.
If we choose an $A_2$-configuration $(L_1,L_2)\subset V$ from Lemma~\ref{lem: configuration} such that $\ell_1=L_1\cap S$ and $\ell_2=L_2\cap S$ are the meridian and the longitude (this is possible since $(\ell_1,\ell_2)$ corresponds to the union of zero sections in a plumbing of two disc cotangent bundles over $S^1$), then $\{\ell_1\}$ and $\{\ell_2\}$ are multicurves, and $\{\ell_1,\ell_2\}$ fills $S$, see Figure~\ref{fig: a2surface}.
\begin{figure*}[t]
    \centering
\begin{overpic}[width=15cm]{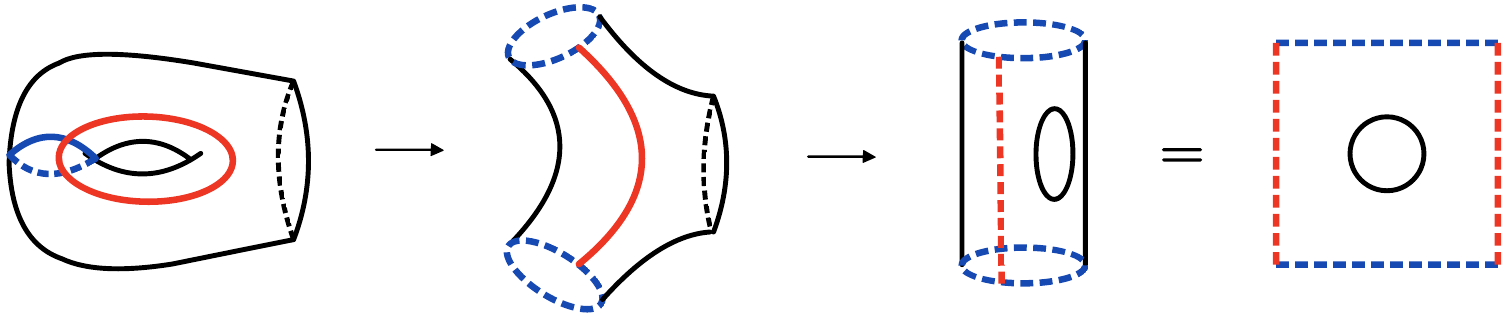}
\put(10, 55){\color{blue} \footnotesize $\ell_1$}
\put(65, 30){\color{red} \footnotesize $\ell_2$}
\put(28, 0){\small $k=2$}

\end{overpic}
    \caption{The complement of the $A_2$-configuration $(\ell_1,\ell_2)$ in the Milnor fibre $S$ for $k=2$ is a punctured disc with 4 edges.
    }
    \label{fig: a2surface}
\end{figure*}

Notice that the only possibilities for two multicurves $\mathcal{C},\mathcal{D}\subset \{\ell_1,\ell_2\}$ such that their disjoint union is $\{\ell_1,\ell_2\}$ and fills $S$ are either $\mathcal{C}=\{\ell_1\}$ and $\mathcal{D}=\{\ell_2\}$, or $\mathcal{C}=\{\ell_2\}$ and $\mathcal{D}=\{\ell_1\}$.
Therefore, every pseudo-Anosov map $\phi_S\in\Symp(S,\lambda_{\st}|_S)$ on $S$ arising from Penner's construction with respect to multicurves whose disjoint union is $\{\ell_1,\ell_2\}$ is of either the form
\begin{equation}\label{eq: phi_S_penner}
	\phi_S=\tau_{\ell_1}^{a_1}\tau_{\ell_2}^{-b_1}\cdots\tau_{\ell_1}^{a_m}\tau_{\ell_2}^{-b_m}
\quad\text{or}\quad \tau_{\ell_2}^{a_1}\tau_{\ell_1}^{-b_1}\cdots\tau_{\ell_2}^{a_m}\tau_{\ell_1}^{-b_m}
\end{equation}
for $m\ge2,\ a_1,b_m\ge0,\ a_2,\dots,a_m,b_1,\dots,b_{m-1}\ge1$, and for $m=1,\ a_1,b_1\geq 1$.
Notice that the latter expression in \eqref{eq: phi_S_penner} is obtained by considering the multicurves $\{\ell_2\}$ and $\{\ell_1\}$.
For example, $\tau_{\ell_1}^2\tau_{\ell_2}^{-2}$ is pseudo-Anosov, whereas $\tau_{\ell_1}^2\tau_{\ell_2}^{2}$ is not.
By Lemma~\ref{lem: lift_dehntwists}, one can lift $\phi_S$ to obtain a map $\phi_V\in\Symp(V,\lambda_{\st};{\bm \sigma})$, which is expressed as either
\begin{equation}\label{eq: A2_pennerLag}
\phi_V=\tau_{L_1}^{a_1}\tau_{L_2}^{-b_1}\cdots\tau_{L_1}^{a_m}\tau_{L_2}^{-b_m}\quad\text{or}\quad
\tau_{L_2}^{a_1}\tau_{L_1}^{-b_1}\cdots\tau_{L_2}^{a_m}\tau_{L_1}^{-b_m}.
\end{equation}
Thus, the class $[\phi_V]\in\pi_0\Symp(V,\lambda_{\st})$ is of $A_2$-Penner-type with respect to $(L_1,L_2)$.
\end{example}

We now define the notion of a Penner-type map on a Liouville domain $(W,\lambda)$ with an $A_k$-configuration $(L_1,\dots,L_k)$ for $k\ge 2$.
Consider an exact symplectic embedding $\Phi\colon (V,\lambda_{\st})\longhookrightarrow (W,\lambda)$ that maps an $A_k$-configuration from Lemma~\ref{lem: configuration} to $(L_1,\dots,L_k)$.
Then the image $\Phi(V)\cong V$ forms a Liouville subdomain of $W$, which we identify with the Milnor fibre of the $A_k$-singularity.
Let $\ell_i=L_i\cap S$ as before.
We say that $\phi\in\Symp(W,\lambda)$ is of \emph{Penner-type} if $\supp(\phi)\subset V\setminus \p V$, the restriction $\phi|_V$ belongs to $\Symp(V,\lambda_{\st};{\bm \sigma})$, and $\phi|_S\in\Symp(S,\lambda_{\st}|_S)$ is given by Penner's construction with respect to two multicurves $\mathcal{C},\mathcal{D}\subset\{\ell_1,\dots,\ell_k\}$ whose disjoint union is precisely $\{\ell_1,\dots,\ell_k\}$, as described in Theorem~\ref{thm: penner_construction}.
A class $[\phi]\in\pi_0\Symp(W,\lambda)$ is said to be of \emph{Penner-type} if it admits a representative $\phi$ of Penner-type.
Conversely, if $\phi_S\in\Symp(S,\lambda_{\st}|_S)$ is given by Penner's construction with respect to multicurves formed from $\{\ell_1\,\dots,\ell_k\}$, then it follows from Lemma~\ref{lem: lift_dehntwists} that we can lift $\phi_S$ to a map $\phi_V\in\Symp(V,\lambda_{\st};{\bm \sigma})$.
Extending $\phi_V$ by the identity map on the complement $W\setminus V$ yields a map $\phi\in\Symp(W,\lambda;V)$ that is of Penner-type.

\begin{remark}\label{rem: twopenner_notions}
The $A_2$-Penner type classes in $(W,\lambda)$ introduced in Section~\ref{sec: intro} are precisely the Penner-type classes for $k=2$ defined above.
This can be verified by adapting the arguments from Example~\ref{ex: pennerA2}, Lemma~\ref{lem: lift_dehntwists} and the preceding discussion.
\end{remark}

\begin{remark}[$A_k$-singularity for $k\ge 2$]\label{rem: akpenner}
Let $S$ denote the two-dimensional Milnor fibre of the $A_k$-singularity for $k\ge 2$.
Then $S$ is a compact Riemann surface of genus $g$ with $m$ boundary circles, where $(m,g)=\displaystyle\Big(1,\frac{k}{2}\Big)$ if $k$ is even, and $(m,g)=\displaystyle\Big(2,\frac{k-1}{2}\Big)$ if $k$ is odd.
As depicted in Figure~\ref{fig:  a3surface},
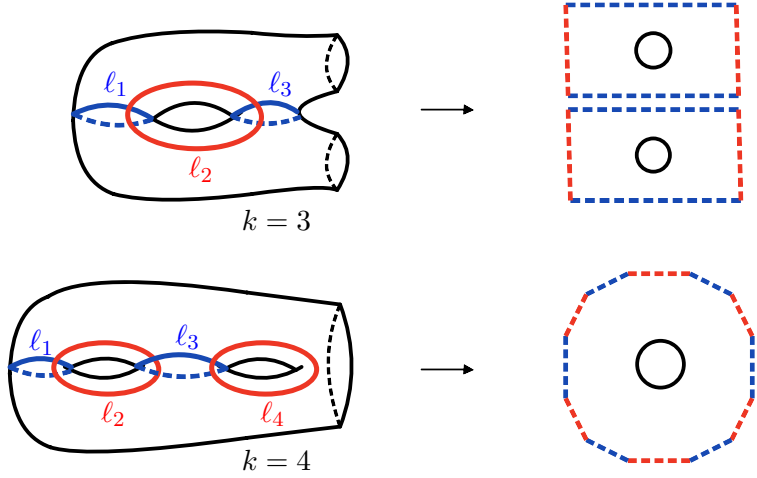
\begin{figure*}[t]
    \centering
\begin{overpic}[width=10cm]{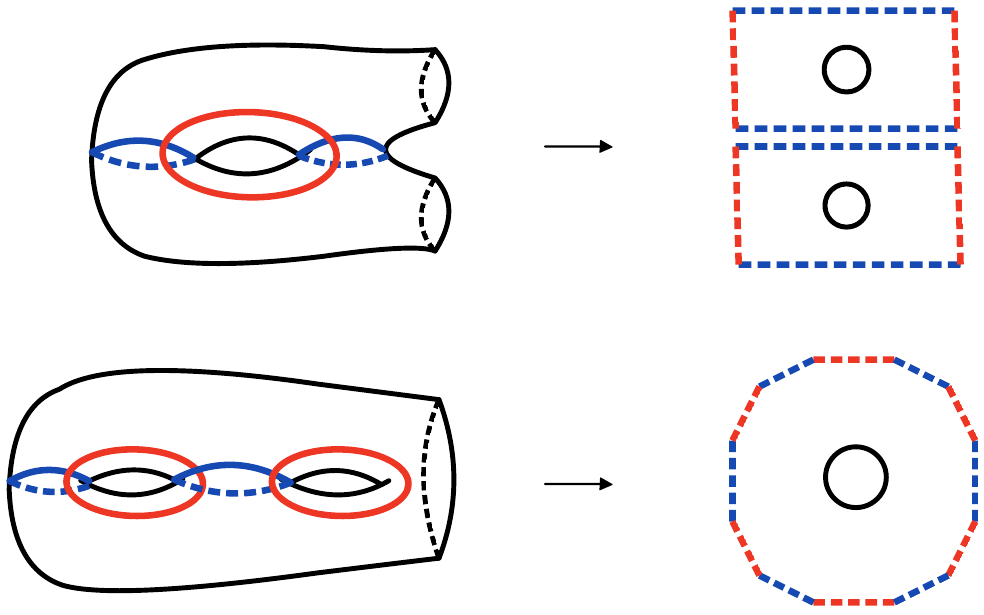}
\put(38, 142){\color{blue} \small $\ell_1$}
\put(100, 142){\color{blue} \small $\ell_3$}
\put(70, 110){\color{red} \small $\ell_2$}
\put(90, 90){\small $k=3$}

\put(90, 0){\small $k=4$}
\put(10, 45){\color{blue} \small $\ell_1$}
\put(65, 47){\color{blue} \small $\ell_3$}
\put(37, 18){\color{red} \small $\ell_2$}
\put(97, 18){\color{red} \small $\ell_4$}

\end{overpic}

    \caption{The complements of the $A_k$-configuration $(\ell_1,\dots,\ell_k)$ in the Milnor fibre $S$ for $k=3,4$ are punctured discs with 8 and 12 edges.
    }
    \label{fig: a3surface}
\end{figure*}
the surface $S$ is diffeomorphic to the plumbing of disc cotangent bundles over $S^1$ associated with the $A_k$ Dynkin diagram.
Let $(\ell_1,\dots,\ell_k)$ be an $A_k$-configuration in $S$ constructed from Lemma~\ref{lem: configuration}.
The zero sections of the disc cotangent bundles in the plumbing constitute an $A_k$-configuration $(\tilde{\ell}_1,\dots,\tilde{\ell}_k)$ of Lagrangian circles.
It is known that the diffeomorphism above can be chosen such that it maps $(\ell_1,\dots,\ell_k)$ to $(\tilde{\ell_1},\dots,\tilde{\ell_k})$, see \cite[Section~3.3.2]{BaKw21} and \cite[Section~1]{Wu14} for details.
Now, the only possibilities for two multicurves $\mathcal{C},\mathcal{D}\subset \{\ell_1,\dots,\ell_k\}$ such that their disjoint union is $\{\ell_1,\dots,\ell_k\}$ and fills $S$ are either
\begin{align*}
\mathcal{C} &= \{\ell_{p}\mid p\le k\text{ is odd}\}\quad \text{and}\quad \mathcal{D}=\{\ell_{p}\mid p\le k\text{ is even}\}, \text{ or}\\
\mathcal{C} &= \{\ell_{p}\mid p\le k\text{ is even}\}\quad\text{and}\quad\mathcal{D}=\{\ell_{p}\mid p\le k\text{ is odd}\}.
\end{align*}
See Figure~\ref{fig: a3surface}.
This motivates the definition of $A_k$-Penner-type classes in Section~\ref{sec: intro}.
In a similar vein as in Remark~\ref{rem: twopenner_notions}, one can show that being an $A_k$-Penner-type class is equivalent to being a Penner-type class with respect to an $A_k$-configuration $(L_1,\dots,L_k)$.
\end{remark}


\section{Proofs of Theorems~\ref{thm: mainthmA}, \ref{thm: mainthmB} and Corollary~\ref{cor: corollaryA}}\label{sec: proofs}
\subsection{Proof of Theorem~\ref{thm: mainthmA}}
As in Section~\ref{sec: penner_milnor}, we identify a Liouville subdomain $V\subset W$ with the Milnor fibre of the $A_k$-singularity.
Let $[\phi]\in \pi_0\Symp(W,\lambda)$ be a class of $A_k$-Penner-type.
It follows from Remark~\ref{rem: akpenner} that $[\phi]$ is of Penner-type.
By definition, we may choose a representative $\phi\in \Symp(W,\lambda;V)$ such that $\phi|_V\in \Symp(V,\lambda_{\st};{\bm \sigma})$ and $\phi|_S\in\Symp(S,\lambda_{\st}|_S)$ is given by Penner's construction formed from $\{\ell_1,\dots,\ell_k\}$.
It follows from Theorem~\ref{thm: penner_construction} that $\phi|_S$ is pseudo-Anosov.
Combining the local-to-global identity (Theorem~\ref{thm: localtoglobal}), the Smith-type inequality (Theorem~\ref{thm: milnor_smith}), and Theorem~\ref{thm: cotton_clay}, we deduce that
$$
h_{\floer}(\phi) = h_{\floer}(\phi|_V) \ge h_{\floer}(\phi|_S) > 0.
$$
This completes the proof.
\subsection{Proof of Corollary~\ref{cor: corollaryA}}\label{sec: corollaryA}
Let $\dim W=2m$.
We recall known results about smooth isotopy classes of Dehn twists when $m$ is even.
First, \cite[Lemma~3.12]{To15} states that every squared Dehn twist $\tau^2$ induces the identity map on homology.
If $m=2$ or $6$, then $\tau^2$ is smoothly isotopic to the identity by a compactly supported isotopy in a neighborhood of the Lagrangian sphere that defines~$\tau$.
More generally, for any even $m$, we can still obtain the same result after passing to some power $\tau^{2i}$, and in fact, $i=4$ is known to be sufficient for all even $m$.
We refer to \cite[Section~5]{Se14_1} and references therein.
Now, suppose that $\dim W=4n$ and that $(L_1,L_2)$ is an $A_2$-configuration of Lagrangian spheres, and consider an $A_2$-Penner-type class $[\tilde{\phi}]\in K(W,\lambda)$ given by
$$
\tilde{\phi}=\tau_{L_1}^{8a_1}\tau_{L_2}^{-8b_1}\cdots\tau_{L_1}^{8a_m}\tau_{L_2}^{-8b_m}
$$
for $m\ge2,\ a_1,b_m\ge0,\ a_2,\dots,a_m,b_1,\dots,b_{m-1}\ge1$, and for $m=1,\ a_1,b_1\geq 1$.
According to Theorem~\ref{thm: mainthmA}, we have $h_{\floer}(\tilde{\phi})>0$, and hence $[\tilde{\phi}]$ has infinite order in $K(W,\lambda)$ by Lemma~\ref{lem: entropy_infinite}.
We shall show that all powers of $[\tilde{\phi}]$ have positive Floer-theoretic entropy, that is, $h_{\floer}(\tilde{\phi}^\ell)>0$ for all $\ell\ge 1$.
While it is generally false that $h_{\floer}(\phi)>0$ implies $h_{\floer}(\phi^\ell)>0$ for $\ell\ge 2$, in our case each power $\tilde{\phi}^\ell$ remains of $A_2$-Penner-type.
Consequently, Theorem~\ref{thm: mainthmA} yields that $h_{\floer}(\tilde{\phi}^\ell)>0$ for all $\ell\ge 1$.

\subsection{Proof of Theorem~\ref{thm: mainthmB}}
Let $\phi\in\Symp(W,\lambda)$ be given.
We begin with the identity
\begin{equation}\label{eq: entropy_prod}
h_{\topo}(\phi) = h_{\topo}(\mathds{1} \times \phi),
\end{equation}
where we view $\mathds{1} \times \phi$ as a map in $\Symp (\widetilde{W\times W},\lambda\oplus(-\lambda))$.
Here $\widetilde{W\times W}$ is a rounded product Liouville domain defined in Section~\ref{sec: lagrfloer}.
Now, \eqref{eq: entropy_prod} follows from \cite[Proposition~3.1.7]{KaHa95}, which states that $h_{\topo}(\phi)=h_{\topo}(\mathds{1}\times \phi)$ holds, considering $\mathds{1}\times\phi$ as a map in $\Diff(W\times W)$ (not on $\widetilde{W\times W}$).
Since the support of $\mathds{1}\times \phi\in\Diff(W\times W)$ is contained in $\widetilde{W\times W}$ (see Remark~\ref{rem: hypersurface_phi}), the identity \eqref{eq: entropy_prod} follows.
The classical Yomdin theorem \cite[Theorem~1.4]{Yo87} implies that the $2n$-dimensional volume growth $v_{2n}(\mathds{1} \times \phi)$ of $\mathds{1} \times \phi\in \Diff(\widetilde{W\times W})$ provides a lower bound for the topological entropy, that is,
\begin{equation}\label{eq: yomdin}
	h_{\topo}(\mathds{1} \times \phi)\ge v_{2n}(\mathds{1} \times \phi).
\end{equation}
Here, $v_{2n}(\mathds{1} \times \phi)$ is defined as the quantity
\begin{equation}\label{eq: volumegrowth}
	v_{2n}(\mathds{1} \times \phi)=\sup_{\Theta}\limsup_{k\to \infty} \frac{\log \Vol((\mathds{1} \times \phi)^k(\Theta))}{k} \in [0,\infty],
\end{equation}
where the supremum runs over all $2n$-dimensional compact submanifolds $\Theta\subset \widetilde{W\times W}$ and $\Vol(\cdot)$ denotes the  volume induced by a Riemannian metric on $W\times W$.
Without loss of generality, we  assume $h_{\floer}(\phi)>0$.
Take $\alpha>0$ with $\alpha<h_{\floer}(\phi)$.
We will show that the diagonal submanifold $\Delta\subset \widetilde{W\times W}$ provides a positive volume growth of $\mathds{1}\times \phi$.
Specifically, we establish that there exists a constant $\tilde{C}>0$ such that
\begin{equation}\label{eq: volume_exp_ineq}
	\Vol((\mathds{1} \times \phi)^k(\Delta))=\Vol(\Gr(\phi^k)) \ge  \tilde{C} e^{\alpha k}
\end{equation}
possibly after passing to a subsequence of $k$.
This will complete the proof in view of \eqref{eq: yomdin} and \eqref{eq: volumegrowth}.

To derive the inequality \eqref{eq: volume_exp_ineq}, we first choose $C'>0$ such that
\begin{equation}\label{eq: floerent_ineq}
\dim \HF(\phi^{k}) \ge C'e^{\alpha k}	
\end{equation}
holds, possibly after passing to a subsequence of $k$.
This follows from the definition of $\alpha$.
Pick a \emph{positive isotopy} $\rho$ of $\Gr(\phi^k)$ with respect to $\Delta$, that is, it is the Hamiltonian 1-flow of a Hamiltonian $H\colon \widetilde{W\times W}\to \R$, where $H$ satisfies the following: it vanishes away from the boundary, is convex near the boundary with respect to cylindrical coordinates, and is linear near the boundary with positive slope strictly less than $\min \Spec(\p(\widetilde{W\times W}),\lambda\oplus(-\lambda))$.
Moreover, the support of $H$ is contained in a region close to the boundary of $\widetilde{W\times W}$ on which $\phi$ acts identically.
The positive isotopy enables us to have $\p \Delta\cap \p (\rho(\Gr(\phi^{k})))=\emptyset$ for all $k\in \N$.
From now on, we write $\Delta'_{k}:=\rho(\Gr(\phi^{k}))$.
By \cite[Lemma~3.2]{BaLe25} 
 and the proof of \cite[Theorem 4.1]{BaLe25}, there exists a \emph{Lagrangian tomograph} of the diagonal $\Delta$, which is a family $\{\Delta_s\}_{s\in B}$ of admissible Lagrangians in $\widetilde{W\times W}$ with $\Delta_0=\Delta$ and boundary $\p \Delta_s=\p \Delta$, parametrised by a closed ball $B$ of sufficiently large dimension, satisfying the following:
\begin{itemize}
	\item $\Delta_s$ is Hamiltonian isotopic to $\Delta$ for all $s\in B$; and
	\item $\Delta_s$ intersects $\Delta_k'$ transversely for each $k\in \N$ and almost every $s\in B$ (depending on $k$). 
\end{itemize}
We also refer to \cite[Section~5.2]{CiGiGu21} for Lagrangian tomographs on closed symplectic manifolds.
For each $k\in \N$ and for almost every $s\in B$ as above, we define the integrable function
$$
N_k(s):=|\Delta_s\cap \Delta_k'|,
$$
where $B$ is equipped with the Lebesgue measure $ds$.
We then deduce that
\begin{alignat*}{3}
\hspace{3cm}	N_{k}(s) &=|\Delta_s\cap \Delta'_{k}| && \\
	&\ge \dim \HF(\Delta_s,\Delta'_{k}) \qquad\qquad && (\text{by definition of $\HF(L_0,L_1)$})\\
	&\ge \dim \HF(\Delta_s,\Gr(\phi^{k}))  && (\text{by invariance property with $\rho$})\\
	&\ge \dim \HF(\Delta,\Gr(\phi^{k}))  && (\text{by invariance property})\\
	&= \dim \HF(\phi^{k})  && (\text{by Theorem~\ref{thm: can_isom}})\\
	&\ge C'e^{\alpha k}. && (\text{by \eqref{eq: floerent_ineq}})
\end{alignat*}
On the other hand, Crofton's inequality in \cite[Lemma~3.4]{BaLe25} provides a constant $C>0$, independent of $\Delta_k'$, such that
$$
\int_B N_k(s)ds \leq C\cdot \Vol(\Delta_k').
$$
Therefore, we establish the inequality
\begin{equation}\label{eq: volume_entropy}
\Vol(\Delta'_{k})\ge \frac{1}{C}\int_B N_{k}(s)ds \ge \tilde{C}e^{\alpha k}
\end{equation}
with the constant $\tilde{C}:=\frac{C'}{C}\int_Bds=\frac{C'}{C}\Vol(B)>0$.
Furthermore, for any sufficiently small $\epsilon>0$ we can choose the positive isotopy $\rho$ such that
\begin{equation}\label{eq: volume}
|\Vol(\Delta'_k)- \Vol(\Gr(\phi^k))|<\epsilon
\end{equation}
holds for all $k\in\N$.
This is possible since the region where $\Delta_k'$ and $\Gr(\phi^k)$ differ is localised near the boundary $\p (\widetilde{W\times W})$.
From \eqref{eq: volume_entropy} and \eqref{eq: volume}, we conclude that
$$
\Vol(\Gr(\phi^k))+\epsilon \ge \tilde{C}e^{\alpha k}
$$
holds for all $k\in \N$.
This completes the proof of Theorem~\ref{thm: mainthmB}.


\subsection*{Acknowledgement}
The authors cordially thank Jungsoo Kang, Felix Schlenk, and Ivan Smith for their interest and for sharing their invaluable insights.
We are especially indebted to Felix Schlenk for a careful reading of an earlier draft of this manuscript.
This work is supported by the Open KIAS Center at Korea Institute for Advanced Study.
JK is supported by the National Research Foundation of Korea (NRF) grants funded by the Korean government (MIST, No. RS-2025-24803252) and through the G-LAMP program (MOE, RS-2024-00441954).
MK is supported by the National Research Foundation of Korea(NRF) grant funded by the Korea government(MSIT) (No. RS-2025-23524132).

\bibliographystyle{abbrv}
\small
\bibliography{mybibfile}

\noindent
\small
{Joontae Kim, Department of Mathematics and Center for Nano Materials, Sogang University, 35 Baekbeom-ro, Mapo-gu, Seoul 04107, Republic of Korea\vspace{0.1cm}\\
\emph{E-mail address: }\texttt{joontae@sogang.ac.kr}\vspace{0.4cm}\\
Myeonggi Kwon, Department of Mathematics Education, and Institute of Pure and Applied Mathematics, Jeonbuk National University, Jeonju 54896, Republic of Korea\vspace{0.1cm}\\
\emph{E-mail address: }\texttt{mkwon@jbnu.ac.kr}
}


\end{document}